\documentclass[12pt]{article}
\usepackage[a4paper, total={6in, 8in}]{geometry}
\usepackage{amsmath}
\usepackage{amsfonts}
\usepackage{amsthm}
\usepackage{hyperref}
\usepackage{bbm}
\usepackage{empheq}
\usepackage{scrextend}
\usepackage{mathtools}
\usepackage{amssymb}
\usepackage{graphicx}

\graphicspath{ {./images/} }
\newcommand*{\p}{\mathbb{P}}
\usepackage{xcolor}
\usepackage{dsfont}
\usepackage{natbib}
\usepackage{graphicx}
\usepackage{cases}
\graphicspath{ {./images/} }
\usepackage{float}
\usepackage{multibib}
 \usepackage{algorithm} 
\newcites{SM}{References}

\usepackage{accents}
\newlength{\dhatheight}
\newcommand{\doublehat}[1]{%
    \settoheight{\dhatheight}{\ensuremath{\hat{#1}}}%
    \addtolength{\dhatheight}{-0.35ex}%
    \hat{\vphantom{\rule{1pt}{\dhatheight}}%
    \smash{\hat{#1}}}}

\newcommand{\norm}[1]{\left\lVert#1\right\rVert}
\usepackage{setspace}
\usepackage{booktabs}

\newtheorem{Lemma}{Lemma}[section]

\newtheorem{theorem}[Lemma]{Theorem}
\newtheorem{corollary}[Lemma]{Corollary}
\newtheorem{remark}[Lemma]{Remark}
\newtheorem{assumption}[Lemma]{Assumption}

\usepackage[utf8]{inputenc}

\theoremstyle{definition}

\date{}

\definecolor{darkblue}{rgb}{.1, 0.1,.8}
\definecolor{darkgreen}{rgb}{0,0.8,0.2}
\definecolor{darkred}{rgb}{.8, .1,.1}

\usepackage{mathtools}
\mathtoolsset{showonlyrefs}
\newcommand*{\E}{\mathbb{E}}
\newcommand*{\V}{\mathbb{V}}
\newcommand*{\U}{\mathbb{U}}
\newcommand*{\W}{\mathbb{W}}
\newcommand*{\N}{\mathbb{N}}
\newcommand*{\R}{\mathbb{R}}
\renewcommand{\P }{{\mathbb P}}
\newcommand{\1}{\mathbbm{1}}

\title{Choosing the Right Norm for Change Point Detection in Functional Data}

\author{
  {\normalsize Patrick Bastian} \\
{\normalsize  Ruhr-Universit\"at Bochum} \\
{\normalsize  Fakult\"at f\"ur Mathematik} \\
{\normalsize  44780 Bochum, Germany}
}

\begin{document}
\maketitle

\begin{abstract}
   We consider the problem of detecting a change point in a sequence of mean functions from a functional time series. We propose an $L^1$ norm based methodology and establish its theoretical validity both for classical and for relevant hypotheses. We compare the proposed method with currently available methodology that is based on the $L^2$ and supremum norms. Additionally we investigate the asymptotic behaviour under the alternative for all three methods and showcase both theoretically and empirically that the $L^1$ norm achieves the best performance in a broad range of scenarios. We also propose a power enhancement component that improves the performance of the $L^1$ test against sparse alternatives. Finally we apply the proposed methodology to both synthetic and real data.
\end{abstract}

\section{Introduction}
\label{sec1}
\textbf{Retrospective Change Point Detection}
In this paper we consider the problem of change point detection in the mean of a functional time series $(X_n)_{n \in \N}$. We denote the mean function at time $n$ by $\mu_n$ and suppose that the data follows the model
\begin{align}
    \mu_1=\mu_2=...=\mu_{k^*}\neq\mu_{k^*+1}=...=\mu_n 
\end{align}
i.e. the time series mean functions have at most one change (AMOC). We denote the mean before and after $k^*$ by $\mu^{(1)}$ and $\mu^{(2)}$, respectively. The problem of interest is then to test
\begin{align}
\label{hypo}
    H_0:\mu^{(1)}=\mu^{(2)} \quad vs \quad H_1:\mu^{(1)}\neq \mu^{(2)}
\end{align}
This problem has received much attention in the literature and we refer the interested reader to the literature review below for further references.\\

\textbf{Cusum Statistics and the choice of norm}
\underline{Cu}mulative \underline{su}m statistics such as
\begin{align}
    \mathbb{U}_n(s,t)=\frac{1}{n}\Big(\sum_{i=1}^{ns}X_i(t)-s\sum_{i=1}^nX_i(t)\Big)
\end{align}
are a central object in the study of testing problems such as \eqref{hypo}. A common approach is to use 
\begin{align}
\label{pb40}
   \hat k&=\arg \max_{s \in [0,1]} \norm{\U(s,\cdot)}\\
   \hat T_n&=\sqrt{n}\max_{s \in [0,1]}\norm{U(s,\cdot)}
\end{align}
as change point estimator and test statistic, respectively. Here $\norm{\cdot}$ is usually the $L^2$ norm. Beginning with \cite{dette2020} there also have been a number of works adopting the supremum norm which allows for more interpretable results particularly in the case of relevant hypotheses.  To the best of our knowledge (see the literature review below) no other norms have been considered for this approach, nor have there been any in depth comparisons between the two choices. The aim of this paper is to close this gap and to propose another alternative that has, in comparison, favorable properties in a number of situations.\\

\textbf{Literature Review}
The literature on change point analysis is vast and a review could easily fill a whole book, we therefore confine ourselves to reviewing change point detection in the context of functional data with a focus on restrospective testing problems.\\

Early work in this area focused mostly on samples of independent curves, \cite{Berkes2009} developed statistical methodology to test for a structural break in the mean function under the assumption of iid errors while \cite{Aue2009} presents results about the asymptotic distribution of a change point estimator in a similar setting. In \cite{Aston2012} these results were extended to weakly dependend time series.  \cite{Zhang2011} propose a self-normalized procedure to test for a structural break in the mean of weakly dependent time series. Extensions for detecting structural breaks in linear models or for detecting smooth deviations from stationarity are available in \cite{Aue2014} and \cite{Delft2017}, respectively. Changes in the (cross-)covariance structure have been considered by \cite{Rice2019},\cite{Stoehr2021} and \cite{Kutta2021}. \\

The above references present their work in the Hilbert space framework and most of them make extensive use of functional principal components analysis (FPCA), i.e. reducing an infinite dimensional problem to a finite dimensional one. This incurs a loss of information which yields inconsistent test procedures when the alternative is part of the orthogonal complement of the principal components. In recent years there have been some efforts to remedy this defect, such as the fully functional change point detection procedures for the mean in \cite{Sharipov2016} and \cite{Aue2017} and for the slope parameter of a functional linear model in \cite{Kutta2022}.\\

The hitherto compiled reference all have on thing in common: They consider the data as random variables in the Hilbert space $L^2$. Since in practice most functions are continuous or even smooth developing a framework based on the space of continuous functions is natural, particularly so as the supremum norm offers interpretability that the $L^2$ norm often lacks. The first works in this direction were \cite{dette2020} and \cite{Dette2022b} where breaks in the mean and covariance functions were considered while \cite{bastian2024} consider gradually occurring changes in the mean. \\

While most works in the subject area focus on a setting with \underline{a}t \underline{m}ost \underline{o}ne \underline{c}hange there are also some works that consider detecting and testing for multiple changes in a functional time series. \cite{Chiou2019} propose a dynamic segmentation and backwards algorithm to detect multiple changes in the mean while \cite{rice2022} extend the classical binary segmentation algorithm to the functional setting. \cite{Harris2022} propose the multiple change isolation method for changes in the mean or covariance structure while \cite{Madrid2022} establish validity for wild binary segmentation in a setting that also allows for sparsely observed data. Finally \cite{bastian2023} provide methodology to detect multiple (relevant) changes for data taking values in the Banach space of continuous functions.

\textbf{Main Contributions}
We provide a brief outline of the main contributions of this paper. 
\begin{enumerate}
    \item[i)] We introduce and argue for the $L^1$ norm as a further alternative to the $L^2$ and supremum norms in the context of change point detection and functional data analysis more generally, focusing on its robustness against heavy tailed data and strong performance against alternatives with dense signal. From a technical point of view we provide a new strong invariance principle for time series taking values in the space of integrable functions under very mild assumptions. 
    \item[ii)] We investigate testing Hypotheses of the form \eqref{hypo} as well as their generalizations given by
    \begin{align}
        H_0(\Delta):\norm{\mu^{(1)}-\mu^{(2)}}_1\leq \Delta \quad vs \quad  H_1(\Delta):\norm{\mu^{(1)}-\mu^{(2)}}_1 >\Delta~. 
    \end{align}    
    Testing the latter is substantially more difficult from a mathematical point of view as $H_0(\Delta)$ does not imply stationarity. An additional challenge is provided by the form of the asymptotic null distribution of the test statistic as it depends on certain level sets of $\mu^{(1)}-\mu^{(2)}$. 
    \item[iii)] We provide an in depth empirical comparison of the properties of the three approaches  based on the $L^1, L^2$ and supremum norm, respectively - both for light and heavy tailed functional data, highlighting the robustness of the $L^1$ norm approach. We further establish theoretical results regarding the asymptotic behaviour of the different procedures under the alternative.
    \item[iv)] We develop a power enhancement component for the $L^1$ procedure that safeguards against sparse alternatives, establish its theoretical validity and demonstrate its performance empirically.
\end{enumerate}

\textbf{Detailed Explanation of Contributions}
Hilbert space methodology based on the $L^2$ norm is the predominant way of modeling in functional data analysis, offering a wide variety of powerful tools to solve statistical problems such as change point detection  (see \cite{ramsay2005}, \cite{Horvath2012} for comprehensive reviews). In recent years modeling functional data as random variables in the space of continuous functions equipped with the supremum norm has also gained some attention due to the natural interpretability of the supremum norm in many practical contexts (see \cite{dette2020}, \cite{bastian2023},\cite{Dette2024}). In this paper we propose $L^1$ based methodology as a further alternative that offers easy interpretability in many contexts (area between curves), robustness against heavy tailed data and great performance for non-sparse alternatives. In section \ref{sec6} we demonstrate empirically that the $L^1$ methodology outperforms both competitors for "dense" alternatives. The $C(T)$ methodology performs better for alternatives that are "sparse" and "spiky" when the data is light tailed but suffers greatly both for "sparse" and "dense" alternatives when the data is heavy tailed. $L^2$ methodology also suffers from heavy tailed data, but not as heavily as the $C(T)$ methodology. It performs worse than the $L^1$ method except in the setting of light tails and a "spiky" alternative. We provide an analysis of the asymptotic behaviour under the alternative for all three norms and further consider a specific class of alternatives to theoretically validate the empirical observations. We also provide theoretical guarantees for our methodology, in particular we establish new weak and strong invariance principles for $L^1$ valued data. Additionally our methods can be used to substantially weaken the assumptions in \cite{dette2020} for $C(T)$ valued data.\\

Similar to the supremum norm the $L^1$ has a natural interpretation in many contexts as it can be interpreted as the area between curves. This makes the $L^1$ norm particularly suitable in the context of relevant hypotheses where one tests hypotheses of the form 
\begin{align}
        H_0(\Delta):\norm{\mu^{(1)}-\mu^{(2)}}_1\leq \Delta \quad vs \quad  H_1(\Delta):\norm{\mu^{(1)}-\mu^{(2)}}_1 >\Delta ~.
\end{align}    
Using these hypotheses in place of the classical hypothesis of exact equality is often a more realistic approach in practice, particularly so as it avoids the problem of detecting arbitrarily small changes that are not practically relevant when the sample size is sufficiently large. These advantages come at a cost however, the analysis of this testing problem is vastly more complicated than the classical setting as the data is not stationary under the null. Additionally, similar to the supremum norm, the $L^1$ norm has only a directional instead of a linear derivative. This results in very complicated limiting distributions under the null and we propose three bootstrap procedures to procure the quantiles required for testing.\\

As mentioned in the first paragraph the $L^1$ norm outperforms the $L^2$ and supremum norm for most settings, to alleviate its poorer performance for "sparse" and "spiky" alternatives we propose a power enhancement component for our proposed methodology that safeguards against these alternatives. Similar ideas have been pursued in a high dimensional setting by \cite{Fan2015}. In the context of functional data \cite{Wang2022} have considered power enhancements against alternatives in certain orthogonal complements. Our proposed enhancement component allows the user to specify the maximum size distortion  that is deemed tolerable as a trade-off for increased power against sparse alternatives.\\

\textbf{Structure of the Paper}
In Section \ref{sec2} we introduce $L^1$ valued random variables accompanied by a strong invariance principle.  Section \ref{sec3} contains a discussion of $L^1$ norm based methodology to detect changes in a functional time series and proposes a test based on a bootstrap procedure. Section \ref{sec4} extends this discussion to relevant hypotheses. Section \ref{sec5} compares $L^1, L^2$ and supremum norm based methodologies from a theoretical point of view and derives their distributions under the alternative. Additionally a power enhancement procedure for the $L^1$ methodology is introduced. Finally Section \ref{sec6}
contains an empirical comparison of the procedures from Section \ref{sec5} and also showcases the performance of the presented methods on real and synthetic data. Proofs can be found in Section \ref{sec7}.

\section{$L^1$ valued random variables}
\label{sec2}
 In this section we provide basic definitions and some statements about central limit theorems and strong invariance principles for random variables with values in the Banach space of integrable functions. We always equip the space
 \begin{align}
     L^1[0,1]=\Big\{ [f] \Big| \int_0^1f(x)dx<\infty, f \in [f]\Big\}
 \end{align}
 with the norm
 \begin{align}
     \norm{f}_1=\int_0^1|f(x)|dx
 \end{align}
 where here and in the following we denote an equivalence class $[f]$ simply by $f$. We also abbreviate $L^1[0,1]$ by $L^1$. We denote the Borel sigma field generated by the open sets with respect to this norm by $\mathcal{B}$ and measurability of random variables taking values in $L^1[0,1]$ is to be understood with respect to this sigma field. Expectations of random variables taking values in $L^1$ are always to be understood as Bochner Integrals, in particular we have for any $X \in L^1$ that 
 \begin{align}
     \E[X] \text{ exists } \iff \E[\norm{X}_1]<\infty~.
 \end{align}
 A centered random variable $X \in L^1$ is said to be Gaussian if and only we have
 \begin{align}
     &(l_1(X), ... , l_k(X))\sim \mathcal{N}(0,\Sigma)         
 \end{align}
 for any finite collection of continuous linear functionals $l_1,...,l_k$ on $L^1$. Here the covariance matrix $\Sigma$ is given by $\Sigma_{ij}=\E[l_i(X)l_j(X)]$. Note that linear functionals on $L^1$ are given by integration against essentially bounded functions. To any such variable (and more generally any random variable for which $\E[X^2]$ exists) one can associate a covariance operator
 \begin{align}
     T_X(l,l')=\E[l(X)l'(X)]
 \end{align}
 that operates on $(L^1)^*\times (L^1)^*=L^\infty \times L^\infty$. We henceforth denote by $\mathcal{N}(\mu,T)$ the Gaussian random variable on $L^1$ with mean function $\mu$ and covariance operator $T$, provided that it exists. Not all covariance operators are eligible for Gaussian random variables, see the discussion on pages 260 and 261 in \cite{ledoux1991} - random variables with covariance operators $T$ for which $\mathcal{N}(\mu,T)$ exists are called pregaussian and we denote for a random variable $X$ with covariance $T$ the associated Gaussian by $G_X$. \\

For infinite dimensional Banach spaces the question of whether a CLT holds for random variables with finite second moments is significantly more complicated than in the finite dimensional case. A concerted effort to resolve these issues has resulted in the classification of Banach spaces into type 2 and cotype 2 spaces. In the case of an independent sequence these notions yield satisfying answers (see \cite{ledoux1991}), in particular $L^1$ is a Banach space of Type 1 and cotype 2 so that any iid sequence of  pregaussian distributions enjoys a CLT. Putting aside that allowing for dependence complicates the issue (see \cite{Giraudo2014} for a stationary weakly dependent sequence of Hilbert space valued variables that does not satisfy a CLT) we note that a CLT is not enough for our purposes, deriving limit theorems for CUSUM statistics typically necessitates weak convergence results for the partial sum process, i.e. weak or strong invariance principles which do not hold in the same generality even for Hilbert spaces (see e.g. \cite{Dehling1983b}). The situation is even more complicated in this case and while there are some results available (confer \cite{Dehling1983}, \cite{Samur1987}) they do not provide ready to use results for the space $L^1$. In order to obtain a strong invariance principle  for $L^1$ valued weakly dependent sequences we impose the following mild assumptions.
\begin{assumption}
The random variables
\begin{align}
    X_{n,i}=\mu_{n,i}+\epsilon_i
\end{align}
form a triangular array of $L^1[0,1]$ valued random variables where $\epsilon_i$ is a mean zero stationary sequence.
    \begin{enumerate}
        \item[A1)] We have 
            \begin{align}
                \E[\norm{\epsilon_i}_1^{2+\delta}]<\infty
            \end{align}
            for some $\delta>0$. 
        \item[A2)] The sequence $(\epsilon_i)_{i \in \N}$ is $\beta$-mixing with coefficients $\beta(k)$ that satisfy
        \begin{align}
            \beta(k) \leq k^{-(1+\epsilon)(1+2/\delta)}, \quad  \sum_{k=1}^\infty k^{1/(1/2-\tau)}\beta(k)^{\delta/(2+\delta)}<\infty
        \end{align}
        for some $\epsilon>0$ and some $\tau \in (1/(2+2\delta),1/2)$.  We denote the long run covariance operator of $\epsilon_i$ by $C$, provided that it exists. 
        \item[A3)] We have that
        \begin{align}
            \int_0^1 \sqrt{\E[\epsilon_i(t)^2]}dt<\infty
        \end{align}
        \item[A4)] We have for some $\alpha>0$ and $C_1>0$ that
        \begin{align}
            \E[\norm{G_{\epsilon_1}(\cdot)-G_{\epsilon_1}(\cdot+y)}_1^2]^{1/2}\leq C_1y^\alpha
        \end{align}
        
    \end{enumerate}
\end{assumption}
Before we state the main result of this section we note that for any Gaussian distribution $\mathcal{N}(0,T)$ on $L^1$ one may define a $L^1$-valued Brownian motion $(B(t))_{t\geq 0}$ that is characterized by $B(1)\sim \mathcal{N}(0,T)$. 
\begin{theorem}
\label{Thm:L1Inv}
   Suppose that assumptions A1) to A4) hold. We then have that, on a possibly larger probability space, there exists a $L^1$ valued Brownian motion $B$ with covariance operator $C$ such that
   \begin{align}
       \norm{\sum_{j \leq t}\epsilon_j - B(t)}_1 \leq t^{1/2-\gamma}
   \end{align}
   for some $\gamma>0$.
\end{theorem}
\textbf{Discussion of the Assumptions}
Assumptions in the vein of A1) and A2), i.e. sufficient moments and decaying dependence coefficients, are standard (see for instance \cite{Sharipov2016}, \cite{Aue2017}, \cite{dette2020}) and ensure that sample means of dependent variables concentrate in a $n^{-1/2}$ neighborhood of their expectations. Assumption 3 is a necessary and sufficient condition for an $L^1$ valued sequence of iid random variables to satisfy a CLT and is therefore the weakest possible. Regarding assumption A4) we note that by the Kolmogorov-Riesz Theorem the modulus of continuity of $\norm{G_{\epsilon_1}(\cdot)-G_{\epsilon_1}(\cdot+y)}_1^2$ is finite almost surely. A4) is therefore a mild quantitative tightness requirement which one may, for instance, verify by checking if the covariance function 
\begin{align}
    K(s,t)= \E[G_{\epsilon_1}(s)G_{\epsilon_1}(t)]=\E[\epsilon_1(s)\epsilon_1(t)]
\end{align}
is (piecewise) smooth (this function is well defined by Theorem 10.12 in \cite{Janson2012}). One can also further weaken this assumption, in that case we will only obtain a weak instead of a strong invariance principle (which still suffices for all the other results in this paper).

\begin{remark}
{\rm
\begin{itemize}\leavevmode
    \item[i)]
    In the proof of Theorem \ref{Thm:PowImp} we also establish a strong invariance principle for $\beta$-mixing random variables in the space $C([0,1])$, this improves on the results in \cite{dette2020} where a weak invariance principle is provided for $\phi$-mixing data. It can also be easily generalized to any totally bounded space.
    \item[ii)] In the proofs section we also provide a weak invariance principle for a bootstrap version of the sequential sum process. The proof can be adapted to also work for the original data and goes through with the weaker assumption of $\alpha$-mixing. All following results in this paper only require a weak invariance principle and are therefore also valid for this weaker dependence concept!
    \item[iii)] While the strong invariance principle we just presented relies on ($\beta$-)mixing as a dependence concept the other results of this paper are also valid for other dependence concepts such as $L^p$-$m$-$approximability$.
    \item[iv)] Functional time series fulfilling the mixing assumptions in A2) include Markov Chains and AR processes (see Section 4 in \cite{Lu2022}).
    \item[v)] The restriction to functions on the interval $[0,1]$ is merely a matter of convenience, any rescaling of the interval will lead to analogous results to those we present in this paper. It is also straightforward to extend the results to multivariate functions on products of intervals.
\end{itemize}
    }
\end{remark}

\section{Change Point Detection: Classical case}
\label{sec3}
In this section we consider a triangular array
\begin{align}
    X_{n,i}=\mu_{n,i}+\epsilon_i \quad i=1,...,n
\end{align}
of $L^1$ valued random variables and write $\mu_i, X_i$ instead of $\mu_{n,i}, X_{n,i}$ for easier reading. We assume that there exists an index $k^*=\lfloor ns^*\rfloor$, where $s^* \in (0,1)$, for which it holds that
\begin{align}
\label{setting}
    \mu^{(1)}=\mu_1=...=\mu_{k^*}, \quad \mu^{(2)}=\mu_{k^*+1}=...=\mu_n
\end{align}
and want to construct an asymptotic level $\alpha$ test for the hypotheses
\begin{align}
\label{classhyp}
    H_0:\mu^{(1)}=\mu^{(2)} \quad vs \quad H_1:\mu^{(1)} \neq \mu^{(2)}~.
\end{align}
Here and in the remainder of this paper we assume that the full trajectories of the observations are available. The results we present can be extended in a straightforward manner to time series that are observed on a sufficiently dense (random) grid as long as one can calculate the $L^1$ norm of the resulting partial sum process with error of order $o(n^{-1/2})$.\\

We now return to constructing a test for the hypotheses \eqref{classhyp} and to that end we define the cusum process by
\begin{align}
    \mathbb{U}_n(s)=\frac{1}{n}\Big(\sum_{i=1}^{ns}X_i-s\sum_{i=1}^nX_i\Big)=S_n(s)-sS_n(1)~
\end{align}
where the sequential process $S_n$ is defined in the obvious way. 
Here and in the remainder of the paper sums with non-integer bounds are understood to be linearly interpolated. $\U_n(s)$ therefore takes values in $C([0,1],L^1)$ equipped with the norm $\norm{f}_{\infty,1}:=\sup_{s \in [0,1]}\norm{f(s)(\cdot)}_1$. Motivated by the fact that the norm of $\U_n(s)$ will be large at the true (rescaled) breakpoint our test statistic and a change point estimator are given by
\begin{align}
\label{pb1}
    \hat T_n&=\sqrt{n}\max_{s \in [0,1]}\norm{\U_n(s)}_1\\
    \hat k&=n\arg \max_{s \in [0,1]} \norm{\U_n(s)}_1=:n\hat s~,
\end{align}
respectively. By the rescaling properties of Brownian motion and Theorem \ref{Thm:L1Inv} we obtain the asymptotic distributions of $\U_n$ and $\hat T_n$.
\begin{corollary}
    Under assumptions A1) to A4) and when $H_0$ holds we have
    \begin{align}
        \{\sqrt{n}\U_n(s)\}_{s \in [0,1]} \overset{d}{\rightarrow} \{\W(s)-s\W(1)\}_{s \in [0,1]}
    \end{align}
    where $\W$ is an $L^1$ valued Brownian motion with covariance $C$. In particular it follows from the continuous mapping theorem that
    \begin{align}
        \hat T_n \rightarrow \norm{\W(s)-s\W(1)}_{\infty,1}
    \end{align}
\end{corollary}
Regarding the change point estimator $\hat k$ we obtain the following result
\begin{theorem}
    \label{Thm:CPCon}
    Grant Assumptions $(A1)$ and $(A2)$ and assume that $\mu^{(1)}\neq \mu^{(2)}$. We then have for $\hat s$ defined in \eqref{pb1} that 
    \begin{align}
        |\hat s - s^*| &=O_\p(n^{-1})\\
        |\hat k - k^*| &=O_\p(1)
    \end{align}
\end{theorem}

As the limiting distribution of $\hat T_n$ is data dependent we propose a bootstrap procedure to access its quantiles. To that end define
\begin{align}
    \hat \mu^{(1)}&=\frac{1}{\hat k}\sum_{i=1}^{\hat k}X_i\\
    \hat \mu^{(2)}&=\frac{1}{n-\hat k}\sum_{i=\hat k+1}^nX_i\\
    \hat Y_{n,i}&=X_{n,i}-(\hat \mu^{(2)}-\hat \mu^{(1)})\1\{i \geq \hat k\}
\end{align}
and let $(\nu_i)_{i \in \N}$ be a sequence of iid standard normal random variables. Let $l=l_n$ be a sequence of natural numbers satisfying $l_n \simeq n^\beta$ where $\beta \in (1/5,2/7)$ and
    \begin{align}   
    \label{bandwidth}
      \frac{\beta(2+\delta)+1}{2+2\delta }<\tau < 1/2~.
    \end{align}
The bootstrap version of the partial sum process is then given by
\begin{align}
\label{bseqdef}
  S_n^*(s)=\frac{1}{n}\sum_{i=1}^{ns}\frac{\nu_i}{\sqrt{l}}\Big(\sum_{k=0}^{i+l-1}\hat Y_{n,i+k}-\frac{l}{n}\sum_{j=1}^n \hat Y_{n,j}\Big)
\end{align}
from which we then obtain the bootstrap version of $\U_n$ and $\hat T_n$ by
\begin{align}
\label{bootdefclassic}
    \U_n^*(s)&=S_n^*(s)-sS_n^*(1)\\
    \hat T_n^*&=\sqrt{n}\max_{s \in [0,1]}\norm{\U_n^*(s)}_1~,
\end{align}
respectively. Here we implicitly assume that  for $ s \in [(n-l+1)/n,1]$ we have
\begin{align}
    S_n^*(s)=S_n^*((n-l)/n)~,
\end{align}
and we shall proceed in the same way for all other such gaps in the domain of the functions we define. We denote the $(1-\alpha)$-quantile of $\hat T_n^*$ by $q^*_{1-\alpha}$ and reject $H_0$ whenever
\begin{align}
    \label{btestclassic}
    \hat T_n>q_{1-\alpha}^*~.
\end{align}
We record the asymptotic properties of this test as follows
\begin{theorem}
\label{Thm:BootCons}
    Grant assumptions A1) to A4) and assume that \eqref{bandwidth} holds. 
    Then the test \eqref{btestclassic} 
    \begin{enumerate}
        \item has asymptotic level $\alpha$, i.e. when $H_0$ is true we have
        \begin{align}
            \lim_n \p(\hat T_n>q_{1-\alpha}^*)=\alpha
        \end{align}
        \item is asymptotically consistent, i.e. when $H_1$ is true we have
        \begin{align}
            \lim_n \p(\hat T_n>q_{1-\alpha}^*)=1
        \end{align}
    \end{enumerate}
\end{theorem}

\section{Change Point Detection: Relevant case}
\label{sec4}
We adopt the setting and notation from the previous section about change point detection in the classical case, the main difference is that we will be testing the hypotheses
\begin{align}
\label{h1}
      H_0(\Delta):\norm{\mu^{(1)}-\mu^{(2)}}_1\leq \Delta \quad vs \quad  H_1(\Delta):\norm{\mu^{(1)}-\mu^{(2)}}_1 >\Delta 
\end{align}
instead. This is motivated by the fact that in real world applications hypotheses like \eqref{classhyp} that assert exact equality are too optimistic - or as \cite{tukey1991} puts it : {\it `` Statisticians classically asked the wrong question - and were willing to answer with a lie, one that was often a downright lie. They asked ''Are the effects of A and B different?'' and they were willing to answer ''no''.``} This is particularly relevant in a modern context where sample sizes can be so large that even small but practically irrelevant changes are picked up by consistent testing procedures. The proposed hypotheses \eqref{h1} avoid this problem by means of allowing for "small" changes under the null, here "small" has a precise meaning in the form of the threshold $\Delta$ and can either be specified by the user or by a data driven procedure that we describe in remark \ref{r1}. As an example we mention the analysis in \cite{bastian2023} of a biomechanical data set, there the authors mapped changes in the mean of knee angle data (i.e. each observation corresponds to the observed knee angles of one stride) during a prolonged running period to fatigue statues of the runner. The authors also used relevant hypotheses to discard negligible changes in the runners gait that were not associated with fatigue. \\

Hypotheses of this kind are particularly interesting when considering the $L^1$ or supremum norm as these norms offer a natural interpretation for the value of $\Delta$ that is accessible to practicioners (area between the curves, maximum absolute deviation between the curves).

\textbf{Definition of the Test}
A straightforward calculation establishes that
\begin{align}
    \E[\U_n(s)]=(s \land s^*-ss^*)(\mu^{(1)}-\mu^{(2)})~.
\end{align}
As the function $s \rightarrow (s \land s^*-ss^*)$ on the interval $[0,1]$ attains its maximum at $s=s^*$ the statistic
\begin{align}
    \hat T_{n,\Delta}=\sqrt{n}\Big(\sup_{s \in [0,1]}\norm{\U_n(s)}_1-\hat s(1-\hat s)\Delta \Big)
\end{align}
is a natural candidate to test the hypotheses \eqref{h1}. More specifically, letting $d_1=\norm{\mu^{(1)}-\mu^{(2)}}_1$, we will establish the following result regarding its asymptotic behaviour:
\begin{theorem}
\label{Thm:BootTest}
Let $\Delta>0$ and grant assumptions A1) to A4). Then
    \begin{align}
\hat T_{n,\Delta}\overset{d}{\rightarrow}
    \begin{cases}
        - \infty \quad &d_1 <\Delta\\
        T \quad &d_1=\Delta\\
        \infty \quad &d_1>\Delta
    \end{cases}
\end{align}
Here $T$ is distributed as
\begin{align}
    T\overset{d}{=}\int_{\mathcal{N}^c}\text{sgn}(d(t))(\W(s^*)-s^*\W(1))(t)dt+\int_{\mathcal{N}}\Big|(\W(s^*)-s^*\W(1))(t)\Big|dt
\end{align}
where
\begin{align}
    d(t)&=\mu^{(1)}(t)-\mu^{(2)}(t)\\
    \mathcal{N}&=\{ t \in [0,1] \ | \ d(t)=0\}~.
\end{align}
\end{theorem}
Using the quantiles of $T$ would therefore yield a consistent asymptotic level $\alpha$ test for $\eqref{h1}$. Unfortunately this test is not feasible as it depends on the data in a rather complicated manner, to be precise it depends both on the covariance structure of the error process and on the set $\mathcal{N}$. We propose three bootstrap procedures to cope with this issue, a comparison of their performances and recommendations regarding their use is given in Section \ref{sec6}. \\

\textbf{Procedure 1:}
The first procedure is based on directly estimating the limiting distribution, to that end we define
\begin{align}
    \hat d(t)&=\hat \mu^{(1)}(t)-\hat \mu^{(2)}(t)\\
    \hat{\mathcal{N}}&=\Big\{ t \in [0,1] \ | \ |\hat d(t)|\leq \frac{\log(n)}{\sqrt{n}}\Big\}
\end{align}
and let
\begin{align}
\label{boot1}
    \hat T^*=\int_{\mathcal{N}^c}\text{sgn}(\hat d(t))\U^*_n(\hat s,t)dt+\int_{\mathcal{N}}\Big|\U^*_n(\hat s,t)\Big|dt~.
\end{align}
Letting $\hat q^*_{1-\alpha}$ denote the $(1-\alpha)$-quantile of $\hat T^*$ we obtain that 
\begin{theorem}
\label{Thm:RelBoot}
    Grant assumptions A1) to A4) and assume that \eqref{bandwidth} holds.  Then we have that
    \begin{align}
        \sqrt{n}\hat T^* \overset{d}{\rightarrow} T
    \end{align}
    conditionally on $X_1,...,X_n$ in probability. In particular the test that rejects $H_0(\Delta)$ whenever
    \begin{align}
        \1\{ \hat T_{n,\Delta}>\hat q^*_{1-\alpha}\}
    \end{align}
    is consistent and has asymptotic level $\alpha$.
\end{theorem}

\textbf{Procedure 2:}
The second test is based on the observation that in applications it is quite common that
\begin{align}
    \lambda(\mathcal{N})=0
\end{align}
holds. In that case procedure 1 will be conservative in finite samples as $\lambda(\hat{\mathcal{N}})$ is larger than 0 whenever the two curves $\mu^{(1)}$ and $\mu^{(2)}$ intersect at least once. We therefore propose 
\begin{align}
\label{boot2}
    \doublehat{T}^*=\int_0^1\text{sgn}(\hat d(t))\U^*_n(s,t)dt
\end{align}
as an alternative to $\hat T^*$. This test naturally only holds the desired nominal level in the case that $\lambda(\mathcal{N})=0$ (we omit a precise statement, it follows along the same lines as Theorem \ref{Thm:RelBoot}).\\

\textbf{Procedure 3:}
The last procedure is based on a standard bootstrap statistic of the form
\begin{align}
\label{boot3}
    \tilde T_n^* = \int_0^1\text{sgn}|\U^*_n(s,t)|dt
\end{align}
which is an obvious upper bound for $\hat T^*$, it is easy to show that an analogue of Theorem \ref{Thm:RelBoot} holds with the caveat that the test will be conservative whenever $\lambda(\mathcal{N})<1$. This is always the case whenever $d_1>0$ and so the test is always conservative. We again omit a precise statement for the sake of brevity. The advantage of this statistic is that no estimation of the $\mathcal{N}$ is necessary which, in practice, requires specifying a data dependent multiplier of the rate $\log(n)/\sqrt{n}$ in the definition of its estimator.

\begin{remark} \leavevmode
{\rm
\label{r1}
From a practical point of view the question of how to choose $\Delta$ is of utmost importance as reasonable choices depend on the application at hand. In some cases real world demands or subject knowledge of the practitioner may yield a natural choice of $\Delta$, yet in many cases such information is hard to come by. Fortunately there is, for fixed nominal level $\alpha$, a natural mechanism to determine a threshold $\Delta$ from the data that then can serve as a measure of evidence against the assumption of a constant mean.\\

This mechanism is based on the observations that the hypotheses $H_0(\Delta)$ in  \eqref{h1} are nested, that the test statistic $\hat T_{n,\Delta}$ is monotone in $\Delta$ and that the bootstrap quantiles do not depend on $\Delta$. Rejecting $H_0(\Delta)$ for some $\Delta>0$ thus also implies rejection $H_0(\Delta')$ for all $\Delta'<\Delta$, hence the sequential rejection principle allows for simultaneous testing without incurring size distortions due to multiple testing. More precisely we may test the hypotheses \eqref{h1} for larger and larger $\Delta$ until we find the minimal $\Delta$ for which the hypothesis $H_0(\Delta)$ is not rejected, i.e.
\begin{align}
       \hat \Delta_\alpha:=\min \big \{\Delta \ge 0 \,| \, \hat T_{n,\Delta}\leq q^*_{1-\alpha} \big  \}=\big(\hat d_{\infty,n}-q^*_{1-\alpha}(nh_n)^{-1/2}\big)\lor 0~.
\end{align}  
}
\end{remark}

\section{Power Comparison and Enhancement}
\label{sec5}

In the following we now compare three bootstrap tests for the hypotheses \eqref{classhyp} based on the $L^1,L^2$ and supremum norm, respectively. To be more precise we consider the test \eqref{btestclassic} based on the $L^1, L^2$ and supremum norm, i.e. we define
\begin{align}
    \hat T_n^{(1)}&=\sqrt{n}\max_{s \in [0,1]}\norm{\U_n(s)}_1\\
     \hat T_n^{(2)}&=\sqrt{n}\max_{s \in [0,1]}\norm{\U_n(s)}_2\\
     \hat T_n^{(\infty)}&=\sqrt{n}\max_{s \in [0,1]}\norm{\U_n(s)}_\infty\\
    W_1&=\sup_{s \in [0,1]}\norm{\W(s)-s\W(1)}_{1}\\
    W_2&=\sup_{s \in [0,1]}\norm{\W(s)-s\W(1)}_{2}\\
    W_\infty&=\sup_{s \in [0,1]}\norm{\W(s)-s\W(1)}_{\infty}
\end{align}
and note that by (minor variations of) the results of \cite{Sharipov2016}, \cite{dette2020} and of this paper we have that $\hat T_n^{(i)}\overset{d}{\rightarrow} W_i$ for $i=1,2,\infty$ under suitable assumptions. The same holds for the respective bootstrap versions of the statistics.\\

Let us now consider a fixed alternative, i.e. we have 
\begin{align}
    \mu^{(1)}-\mu^{(2)}=\Phi~
\end{align}
for some fixed function $\Phi$ which, to avoid cumbersome technical discussions, we shall assume to be continuous. We consequently define the three tests
\begin{align}
    \1\{\hat T_n^{(1)} \geq q^*_{1-\alpha, 1} \}\\
    \1\{ \hat T_n^{(2)}  \geq q^*_{1-\alpha, 2}\}\\
    \1\{ \hat T_n^{(\infty)}  \geq q^*_{1-\alpha, \infty}\}\\
\end{align}
where $q^*_{1-\alpha, i}$ is given by the quantile of the respective bootstrap statistic.\\
A simple application of Jensen's inequality yields that
\begin{align}
\label{pb20}
   q^*_{1-\alpha, 1} < q^*_{1-\alpha, 2} < q^*_{1-\alpha, \infty}~.
\end{align}
We also have the following theorem regarding the behaviours of $\hat T_n^{(i)}$ under the alternative.
\begin{theorem}
\label{Thm:Alter}
    Grant assumptions A1) to A3) for the supremum (and hence also for the $L^1$ and $L^2$) norm, further assume that $X_{n,i}$ takes values in $C([0,1])$ and that \begin{align}
        \E[|\epsilon_i(s)-\epsilon_i(t)|^2]^{1/2} \lesssim |s-t|^\alpha~.
    \end{align}
    for some $\alpha>1/2$. Then
    \begin{align}
        &\hat T_n^{(1)}-\sqrt{n}\norm{\Phi}_1 \rightarrow A_1:=T\\
         &\hat T_n^{(2)}-\sqrt{n}\norm{\Phi}_2 \rightarrow A_2:=\langle \W(s^*)-s^*\W(1), \Phi/\norm{\Phi}_2 \rangle\\
          &\hat T_n^{(\infty)}-\sqrt{n}\norm{\Phi}_\infty \rightarrow A_\infty:=\sup_{t \in \mathcal{E}}\text{sgn}(\Phi(t))(\W(s^*,t)-s^*\W(1,t))) 
    \end{align}
    where $T$ is given in Theorem \ref{Thm:BootTest} and $\mathcal{E}$ is given by
    \begin{align}
        \{ t \in [0,1] | |\Phi(t)|=\norm{\Phi(t)}_\infty \}~.
    \end{align}
\end{theorem}

Let us now isolate a class of $\Phi$ for which we can compare the performances of the three tests reasonably well. To that end define
\begin{align}
    \mathfrak{A}:=\{ \Phi_c:[0,1] \rightarrow \R | \Phi(t)=e^{-c(x-0.5)^2}, c\in [0,\infty) \}
\end{align}
and note that $c$ is a sparsity parameter, the higher $c$ the sparser the signal is distributed over the interval. 

Letting $\Phi=\Phi_c$ we therefore obtain
\begin{align}
\label{pb21}
    \p(\hat T^{(i)}>q^*_{1-\alpha,i})\simeq \p(A_i>q_{1-\alpha,i}-\sqrt{n}\norm{\Phi_c}_i)
\end{align}
where $q_{1-\alpha,i}$  is the $(1-\alpha)$-quantile of $W_i$.\\
\begin{figure}[H]
 
    \includegraphics[scale= 0.66]{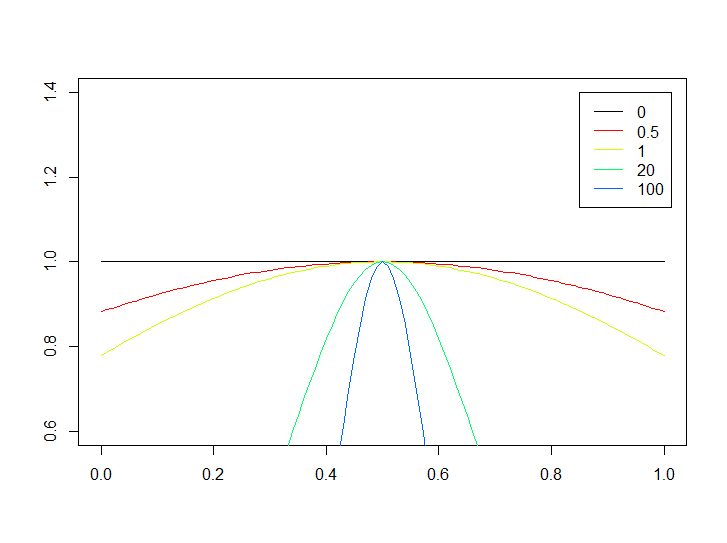}       
    
    \caption{\it  Plots of $\Phi_c$ for different choices of $c$. 
 }
    \label{Fig:2}
\end{figure}
It is easy to see that $A_1,A_2$ and $A_\infty$ are continuous (in probability) in $c$ and the same clearly holds for $\norm{\Phi_c}_i, i=1,2,\infty$. Further we observe that $A_1$ is increasing in $c$ while $A_\infty$ is decreasing in $c$. Let us consider the boundary points $c=0$ and $c\rightarrow \infty$. In the first case we have that 
\begin{align}
    A_1=A_2\leq A_\infty\\
    \norm{\Phi_c}_1=\norm{\Phi_c}_2=\norm{\Phi_c}_\infty=1    
\end{align}
while in the second case we have that $A_1,A_2,A_\infty$ are uniformly subgaussian (equation 3.2 in \cite{ledoux1991}) while for $c \rightarrow \infty$
\begin{align}
    &\norm{\Phi_c}_1=o(1)\\
    &\norm{\Phi_c}_2=o(1)\\
    &\norm{\Phi_c}_\infty=1
\end{align}
Combining this with equation \eqref{pb20} therefore yields
\begin{corollary}
Based on equation \eqref{pb21} we define
\begin{align}
    Pow(i,c)=\p(A_i>q_{1-\alpha,i}-\sqrt{n}\norm{\Phi_c}_i)
\end{align}
and obtain, for c sufficiently large, that
\begin{align}
    Pow(1,0)>Pow(2,0)>Pow(\infty,0)\\
    Pow(1,c)+Pow(2,c)<Pow(\infty,c)
\end{align}
This mirrors the results \cite{He2021} have obtained regarding one sample covariance testing in a high dimensional regime (see section 2.2 in their paper). 
\end{corollary}
This indicates that the $L^1$ norm is the best choice for dense signals while the supremum norm is superior for sparser alternatives. What happens in intermediate settings depends heavily on the long run covariance $C$ and is beyond the scope of this paper. The simulations in section \ref{sec6} will elucidate some of the possible outcomes. \\

\textbf{Power enhancement}
As we have just seen the test \eqref{btestclassic} based on the statistic $\hat T_n$ suffers from a lack of power against alternatives where the mean difference is tightly concentrated in a small area. We will alleviate this issue by means of a power enhancement component similar in spirit to that in \cite{Fan2015}. Let us recall their definition of such a component. We call a statistic $J$ a power enhancement component when it satisfies the three conditions below
\begin{enumerate}
    \item[I)] $J\geq 0$ almost surely.
    \item[II)] Under the null we have $J=0$ with high probability (or at least $J=o_\p(1)$).
    \item[III)] In some region of the alternative we have $J \rightarrow \infty$ in probability.
\end{enumerate}

Due to II) the asymptotic null distribution is not changed by considering $\hat T_n+J$ instead of $\hat T_n$. In the following we will propose a power enhancement component that  fulfills I) to III) and allows for a user specified probability of the failure of the condition $J=0$ under the null, i.e. a specification of the tolerated size distortion of the test. \\

We define for some sequence $\eta_n$
\begin{align}
    J_n=\sqrt{n}\norm{\U_n(\hat s, \cdot)}_\infty \1\{\sqrt{n}\norm{\U_n(\hat s, \cdot)}_\infty \geq \eta_n\}~.
\end{align}
and observe that under suitable assumptions one can show that 
\begin{align}
    \sqrt{n}\norm{\U_n(\hat s, \cdot)}_\infty \rightarrow \norm{\W(s^*,\cdot)}_\infty:=\mathbb{T}
\end{align}
This suggests choosing 
\begin{align}
    \eta_n=q_{1-\alpha_n}^J
\end{align}
where $q_{1-\alpha}^J$ is the $(1-\alpha)$-quantile of $\mathbb{T}$ and $\alpha_n=o(1)$ is a sequence that describes the maximum size distortion one is willing to suffer as a trade-off for the increased detection power for sparse alternatives. We record the theoretical properties of $J$ below.
\begin{theorem}
    \label{Thm:PowImp}
    Grant assumptions A1) and A2) for the supremum instead of the $L^1$ norm.  Further assume that \eqref{bandwidth} holds. Additionally we require that $X_{n,i} \in C([0,1])$ and that
    \begin{align}
        \E[|\epsilon_i(s)-\epsilon_i(t)|^2]^{1/2} \lesssim |s-t|^\alpha~.
    \end{align}
    for some $\alpha>1/2$.   Let $\alpha_n$ be any sequence tending to 0 and let $\eta_n=q^J_{1-\alpha_n}$, then the test
    \begin{align}
        \1\{\hat T_n+J_n \geq \hat q^*_{1-\rho}\}
    \end{align}
    is consistent and has asymptotic level $\rho$.
\end{theorem}
\begin{remark}
{\rm

    \begin{enumerate}\leavevmode
        \item[i)] The restriction $\alpha>1/2$ can be dropped at the cost of making assumptions A1) and A2) slightly stronger. We omit this for the sake of a parsimonious presentation.
        \item[ii)] One can extend this result to piecewise continuous functions with deterministic locations of discontinuity in a straightforward manner.
        \item[iii)] The quantiles $q^J_{1-\alpha_n}$ can be accessed by means of a bootstrap procedure, to be precise we use  the same set up as in \eqref{bootdefclassic} but take the supremum norm instead of the $L^1$ norm.
        \item[iv)] It is natural to ask why one should not just construct a test based on an appropriately rescaled and aggregated test statistic, say of of the form
        \begin{align}
        \label{aggregated}
            \sqrt{n}\max_{i \in \{1,\infty\}}\norm{U_n(\hat s,\cdot)}_i~.
        \end{align}
        The problem in this case is that, different from situations where limits are normally distributed, there is no natural way to standardize the statistics that are aggregated. As the supremum norm of a function is never smaller than its $L^1$ norm the quantiles of the aggregated statistic are entirely determined by the supremum part when no rescaling is applied. As we will see in section \ref{sec6} the bootstrap based on the supremum norm is particularly sensitive to heavy tails and generally performs worse than the $L^1$ norm methodology except in the case of alternatives with sharp (spatial) spikes in the signal. Consequently we advise against using an aggregated statistic of the form \eqref{aggregated} unless one is able to determine a reasonable way to standardize its components.         
    \end{enumerate}
    }
\end{remark}

Let us now consider a sequence of alternatives $H_{1,n}$ against which the original test based on $\hat T_n$ alone lacks power. To that end let
\begin{align}
    \mu^{(1)}(t)=0, \quad \mu^{(2)}(t)=\1\{t \in [0,\beta_n]\}~.
\end{align}
A simple calculation shows that 
\begin{align}
    \norm{\mu^{(1)}-\mu^{(2)}}_1=\beta_n, \quad   \norm{\mu^{(1)}-\mu^{(2)}}_\infty=1
\end{align}
Now if $\beta_n\sqrt{n}\rightarrow 0$ it is easy to see that $\hat T_n$ will still converge to $\norm{\W(s)-s\W(1)}_{\infty,1}$. On the other hand taking $\alpha_n=1/n$ and using the fact that norms of Banach valued Gaussians are Subgaussian (see equation 3.2 in \cite{ledoux1991}) we know that $\eta_n \lesssim \sqrt{\log(n)}$ so that $J_n \rightarrow \infty$ in probability.     \\

In other words: With the power enhancement component we are able to detect changes that happen only in a small spatial region that the original test is not able to detect. The above example is prototypical in the sense that any mean difference concentrated on a small spatial set whose supremum vanishes slowly enough will be detected by $\hat T_n+J_n$ but not by $\hat T_n$ alone. We also demonstrate this empirically in the next section.

\section{Finite Sample Performance}
\label{sec6}
In this section we investigate the finite sample performance of the proposed methodology. In the first subsection we will compare the performances of $L^1,L^2$ and supremum norm methodology on synthetic data sets for independent and dependent data. In the second subsection we investigate the two bootstrap procedures we proposed for the case of relevant hypotheses. Here a direct comparison of the $L^1$ methodology with either the $L^2$ or supremum norm methodology makes little sense as the null hypotheses 
\begin{align}
    H_0(\Delta)^1&=\norm{\mu^{(1)}-\mu^{(2)}}_1\leq \Delta\\
    H_0(\Delta)^2&=\norm{\mu^{(1)}-\mu^{(2)}}_2\leq \Delta\\
    H_0(\Delta)^\infty&=\norm{\mu^{(1)}-\mu^{(2)}}_\infty \leq \Delta
\end{align}
are incomparable. In the third section we apply the proposed methodology to ... 
\subsection{Norm Comparison for the Classical Hypotheses}
In this section we will consider both light and heavy tailed processes. In the independent case we consider
\begin{align}
\label{light} \epsilon_{i,l}&=B_i\\
\label{heavy}    \epsilon_{i,h}&=\sum_{k=1}^{10} f_i t_{ik}
\end{align}
$n \in \{100,200\}$ is the sample size, $B_i$ are iid brownian motions on $[0,1]$, $(t_{ik})$ are an array of iid t distributions with 3 degrees of freedom and $(f_i)_{i=1,...,10}$ is the bspline basis as given in the R package "fda" \cite{fdaR}. In the dependent case we follow \cite{Aue2017} and consider first order functional autoregressions 
\begin{align}
\label{lightdep}     \epsilon_{i,l,d}&=\Psi \epsilon_{i-1,l,d}+\zeta_{i,l}\\
\label{heavydep}     \epsilon_{i,h,d}&=\Psi \epsilon_{i-1,h,d}+\zeta_{i,h}
\end{align}
where the innovation processes are given by
\begin{align}
    \zeta_{i,l}=\sum_{k=1}^{21}N_{i,k}v_kk^{-1}\\
    \zeta_{i,h}=\sum_{k=1}^{21}T_{i,k}v_kk^{-1}
\end{align}
and $N_{i,k}$ and $T_{i,k}$ are given by iid standard normal and t distributions with 3 degrees of freedom, respectively. $v_1,...,v_{21}$ are fourier basis functions on the interval $[0,1]$ while the operator $\Psi$ is given by $\Psi= \Psi_0/\sqrt{2}$ with $\Psi_0$ a $21\times 21$ dimensional matrix whose entries are given by iid standard normals with variances given by $((ij)^{-1})_{1 \leq i,j \leq 21}$. Note that the choice of $\sqrt{1/2}$ as a multiplier for $\Psi_0$ yields a rather strong temporal dependence in the data, this is illustrated by the fact that the estimated bandwidth $l_n$ is typically in the range 6-8 for a sample size of 100.\\
Further we define $s^*=1/2$ and let
\begin{align}
    \mu^{(1)}=0
\end{align}
whereas $\mu^{(2)}$ is given by one of the following choices
\begin{empheq}[left={\mu^{(2)}(t)=\kappa}\empheqlbrace]{align}
    & 0 \label{null} \\ 
    & 1 \label{const} \\ 
    & \sin(\pi t) \label{bump} \\ 
    & \sin(4\pi t) \label{bumps} \\ 
    & 2\exp\big(-100(t-0.5)^2\big) \label{spike}
\end{empheq}

We compare the following three bootstrap procedures: \eqref{btestclassic} for the $L^1$ norm, a multiplier version of the test from \cite{Sharipov2016} for the $L^2$ norm and the test from \cite{dette2020} for the supremum norm, i.e. all three tests are based on norms of the cusum process $\U_n$ (or $\U_n^*$ for the bootstrapped versions). For the choice  of the block length $l_n$ we use the method from \cite{Rice2017} with the recommended quadratic spectral kernel. We display the empirical rejection probabilities for choice $\kappa=0.2$ based on a 1000 bootstrap runs with 200 bootstrap repetitions each in tables \ref{tab1} and \ref{tab2} below.

\begin{table}[H]
	\centering 
	\begin{tabular}{|c|c|c|c|c|c|c|c|c|c|c|}
	    \hline
        \multicolumn{6}{|c|}{Light Tails} & \multicolumn{5}{|c|}{Heavy Tails} \\
        \hline
	$ \mu^{(2)} $ & \eqref{null} & \eqref{const} & \eqref{bump} &\eqref{bumps} & \eqref{spike} & \eqref{null} & \eqref{const} & \eqref{bump} &\eqref{bumps} & \eqref{spike}\\
		\hline            
	 $\norm{\cdot}_1$ & 0.043 &  0.383& 0.180 & 0.070 & 0.080 & 0.028 &  0.252& 0.138 & 0.061 & 0.092 \\
	 $\norm{\cdot}_2$  & 0.046 & 0.295 & 0.155 & 0.084 & 0.099 & 0.023 & 0.200 & 0.106 & 0.062 & 0.093   \\   
      $\norm{\cdot}_\infty$ & 0.041 & 0.176 & 0.096 & 0.178 & 0.282 & 0.026 & 0.044 & 0.022 & 0.036 & 0.046 \\   
		\hline
		\hline
	 $\norm{\cdot}_1$ & 0.049 & 0.646 & 0.301 & 0.090 & 0.143 & 0.034 & 0.505 & 0.260 & 0.138 & 0.162 \\
	 $\norm{\cdot}_2$ & 0.055 & 0.575 & 0.267 & 0.143 & 0.224 & 0.029 & 0.412 & 0.226 & 0.142 & 0.191   \\   
      $\norm{\cdot}_\infty$ & 0.047 & 0.361 & 0.194 & 0.323& 0.664 & 0.030 & 0.150 & 0.056 & 0.059 & 0.116\\   	 	 
		 \hline
	\end{tabular}
	\smallskip
	
	\caption{Empirical Rejection Rates of the bootstrap tests  for the hypotheses \eqref{classhyp} based on different norms for independent data. The error processes are given by \eqref{light} in the light tailed and by \eqref{heavy} in the heavy tailed case. The upper part of the table contains the results for sample size $n=100$, lower part contains the results for sample size $n=200$. In both cases $\kappa=0.2$.}
	\label{tab1}
\end{table}
Let us first consider the setting of independent errors. For the light tailed data we observe that all three tests keep the nominal level and generally behave as predicted by Theorem \ref{Thm:BootCons}. The $L^1$ norm outperforms the other two contenders for the denser alternatives \eqref{const} and \eqref{bump} whereas the supremum norm is the better choice for alternatives \eqref{bumps} and \eqref{spike} that are sparser and spikier. The test based on the $L^2$ method always lies between the two other choices but never outperforms them. This is consistent with the discussion in Section \ref{sec5}. \\
For heavy tailed data the picture changes substantially,  all three tests are slightly conservative and achieve a level of $~0.025-0.030$ instead of the nominal level $0.050$. The $L^1$ method outperforms its competitors across all alternatives except for the choices $\eqref{bumps}$ and $\eqref{spike}$ where the $L^2$ norm achieves a comparable performance and even slightly outperforms the $L^1$ norm for the alternative \eqref{spike} when $n=200$. The performance of the supremum methodology deteriorates drastically in the heavy tailed regime, performing worse than the integral norms even for its most favorable alternative \eqref{spike}.

\begin{table}[H]
	\centering 
	\begin{tabular}{|c|c|c|c|c|c|c|c|c|c|c|}
	    \hline
        \multicolumn{6}{|c|}{Light Tails} & \multicolumn{5}{|c|}{Heavy Tails} \\
        \hline
	$ \mu^{(2)} $ & \eqref{null} & \eqref{const} & \eqref{bump} &\eqref{bumps} & \eqref{spike} & \eqref{null} & \eqref{const} & \eqref{bump} &\eqref{bumps} & \eqref{spike}\\
		\hline            
    	 $\norm{\cdot}_1$ & 0.055 &  0.179& 0.101 & 0.055 & 0.076 & 0.060 &  0.085& 0.049 & 0.051 & 0.058 \\
	 $\norm{\cdot}_2$  & 0.053 & 0.161 & 0.084 & 0.055 & 0.077 & 0.051 & 0.072 & 0.055 & 0.051 & 0.042   \\   
      $\norm{\cdot}_\infty$ & 0.039 & 0.100 & 0.068 & 0.073 & 0.088 & 0.048 & 0.044 & 0.044 & 0.042 & 0.060 \\   
		\hline
		\hline
	 $\norm{\cdot}_1$ & 0.045 & 0.334 & 0.148 & 0.065 & 0.085 & 0.048 & 0.132 & 0.083 & 0.056 & 0.039 \\
	 $\norm{\cdot}_2$ & 0.043 & 0.286 & 0.133 & 0.068 & 0.099 & 0.051 & 0.114 & 0.083 & 0.057 & 0.039   \\   
      $\norm{\cdot}_\infty$ & 0.042 & 0.166 & 0.105 & 0.121 & 0.173 & 0.044 & 0.074 & 0.063 & 0.058& 0.048\\   	 	 
		 \hline
	\end{tabular}
	\smallskip
	
	\caption{Empirical Rejection Rates of the bootstrap tests  for the hypotheses \eqref{classhyp} based on different norms for dependent data. The error processes are given by \eqref{lightdep} in the light tailed and by \eqref{heavydep} in the heavy tailed case. The upper part of the table contains the results for sample size $n=100$, lower part contains the results for sample size $n=200$. In both cases $\kappa=0.2$.}
	\label{tab2}
\end{table}

For dependent data the conclusions are the same with some minor differences. The main one being that all three procedures closely approximate the nominal level even for heavy tails. The power is substantially lower than in the independent setting, which is unsurprising considering the rather large temporal dependence induced by the multiplier $\sqrt{1/2}$ used in the definition of $\Psi$. \\

As a final consideration we investigate the power enhancement procedure defined in Theorem \ref{Thm:PowImp} for light tailed, independent data and the choice $\alpha_n=0.01$. We omit the heavy tailed case because, as demonstrated above, the supremum norm will not contribute to the empirical rejection rate in this case. We focus on the independent case because in the dependent setting we defined the power is generally low and it is hard to distinguish increases in performance and random fluctuations from one another. Similar effects as for the independent case can be observed when increasing the factor 0.2 and 0.4 in the definitions of the alternatives \eqref{const} to \eqref{spike}. We increase the number of bootstrap repetitions for each run from 200 to 1000 to ensure that the $0.01$-quantile of the supremum norm statistic is properly approximated.  The results are summarized in table \ref{tab3} .

\begin{table}[H]
	\centering 
	\begin{tabular}{|c|c|c|c|c|c|c|c|c|c|c|}
	    
        \hline
	$ n $ & \eqref{null} & \eqref{const} & \eqref{bump} &\eqref{bumps} & \eqref{spike} \\
		\hline            
    	 100 & 0.038 &  0.361& 0.175 & 0.078 & 0.114 \\
	 200  & 0.050 & 0.592 & 0.287 & 0.113 & 0.325 \\
		 \hline
	\end{tabular}
	\smallskip
	
	\caption{Empirical Rejection Rates of the $L^1$ bootstrap procedure with power enhancement for the hypotheses \eqref{classhyp}. The error processes are given by \eqref{light}.}
	\label{tab3}
\end{table}

We observe that the approximation of the nominal level is not visibly impacted by the addition of the power enhancement component, similarly the rejection rates for the alternatives \eqref{const} and \eqref{bump} are not visibly impacted. The rejection rates for the alternatives \eqref{bumps} and \eqref{spike} both show an increase that is particularly noticeable for \eqref{spike} which is consistent with the observation that the supremum norm based method performs best for these alternatives as long as the tails are light.\\

\textbf{Recommendation:}
Based on the observations in this section we recommend using the $L^1$ norm methodology as the default option. Among the compared procedures it achieves the best performance under heavy tailed data, performs best for dense alternatives under light tails and one can safeguard against sparse alternatives with light tails via the power improvement component we proposed. In cases where one expects light tails and sparse signals one should use the supremum norm methodology instead.

\subsection{Relevant Hypotheses: Synthetic Data}
We adopt the notation from the previous subsection. As comparing relevant hypotheses for different norms makes little sense (hypotheses of the form \eqref{h1} are neither equivalent nor meaningfully nested when varying the choice of norm) we focus on comparing the three bootstrap procedures \eqref{boot1}-\eqref{boot3} we proposed in section \ref{sec4}. To that end we will again consider $\mu^{(1)}=0$ and $\mu^{(2)}$ given according to equations \eqref{null}-\eqref{spike}, this time for the choice $\kappa =0.4$. We then determine, for each choice of $\mu^{(2)}$ two choices of $\Delta$. More precisely we let
\begin{align}
    \Delta_0&=\norm{\mu^{(1)}-\mu^{(2)}}_1\\
    \Delta_1&=\norm{\mu^{(1)}-\mu^{(2)}}_1/2~,
\end{align}
so that $\Delta_0$ corresponds to testing on the boundary $\Delta=d_1(\kappa)$ and $\Delta_1$ corresponds to testing where the alternative holds. $\Delta_1$ and $\kappa$ are chosen in this way so that the distance to the null is equal to the distance to the null in the alternatives we considered in the classical case. We record the resulting empirical rejection rates for sample sizes $n=100,200, 500$ and independent data (both light and heavy tailed) in tables \ref{tab4} and \ref{tab5} below. To estimate $\mathcal{N}$ we slightly modify $\hat{\mathcal{N}}$ to adjust for the spatially varying noise level, i.e. we use the set estimator
\begin{align}
    \Big\{ t \in [0,1] | |\hat d(t)| \leq \hat \sigma(t)\frac{\log(n)}{\sqrt{n}}\Big\}
\end{align}
instead, here $\hat \sigma(t)$ is the square root of the sample variance of $(X_i(t))_{i=1,...,n}$. As in the previous section the results are based on a 1000 bootstrap runs with 200 bootstrap repetitions each. Similar results hold for dependent data which are omitted for the sake of brevity.

\begin{table}[H]
	\centering 
	\begin{tabular}{|c|c|c|c|c|c|c|c|c|c|c|}
	    \hline
        \multicolumn{5}{|c|}{Light Tails} & \multicolumn{4}{|c|}{Heavy Tails} \\
        \hline
	Statistic  & \eqref{const} & \eqref{bump} &\eqref{bumps} & \eqref{spike} & \eqref{const} & \eqref{bump} &\eqref{bumps} & \eqref{spike} \\
		\hline     
    	\eqref{boot1} &  0.102 & 0.128 & 0.012 & 0.180 &  0.041& 0.063 & 0.017 & 0.142 \\
	 \eqref{boot2} &0.102 & 0.128 & 0.052 & 0.282 & 0.097 & 0.215 & 0.218 & 0.648   \\   
      \eqref{boot3}&  0.045 & 0.063 & 0.002 & 0.101 & 0.009 & 0.020 & 0.003 & 0.061 \\ 
      \hline	
      \hline     
    	\eqref{boot1} &  0.077 & 0.105 & 0.004 & 0.174 &  0.041 & 0.048 & 0.010 & 0.114 \\
	 \eqref{boot2} & 0.077 & 0.106 & 0.034 & 0.279 & 0.077 & 0.162 & 0.147 & 0.599   \\   
      \eqref{boot3}&  0.043 & 0.054 & 0.001 & 0.088 & 0.007 & 0.010 & 0.000 & 0.045\\ 
      \hline	
      \hline     
    	\eqref{boot1} &  0.072 & 0.101 & 0.000 & 0.124 &  0.057 & 0.051 & 0.010 & 0.082 \\
	 \eqref{boot2} &0.072 & 0.101 & 0.036 & 0.215 & 0.068 & 0.115 & 0.112 & 0.500   \\   
      \eqref{boot3}&  0.037 & 0.049 & 0.000 & 0.067 & 0.012 & 0.010 & 0.000 & 0.043 \\ 
      \hline	
	\end{tabular}
	\smallskip
	
	\caption{Empirical Rejection Rates of the bootstrap tests  for the hypotheses \eqref{h1} based on the bootstrap statistics \eqref{boot1}-\eqref{boot3}. The error processes are given by \eqref{light} in the light tailed and by \eqref{heavy} in the heavy tailed case. The upper part of the table contains the results for sample size $n=100$, the middle part for sample size $n=200$ and the lower part contains the results for sample size $n=500$. In all cases $\kappa=0.4$ and $\Delta$ is chosen such that $\Delta=d_1$, i.e. we are on the boundary of the null hypothesis.}
	\label{tab4}
\end{table}

Regarding table \ref{tab4}, i.e. the empirical sizes, we observe that only the method based on \eqref{boot3} keeps the nominal level. The performance of \eqref{boot1} improves with increasing sample size but nonetheless exceeds the nominal level in some cases. This phenomenon can be explained by the fact that the method \eqref{boot1} relies on the delicate task of estimating the set $\mathcal{N}$ which in turn relies on good estimation of the true change point. Change point estimation based on the $L^1$ norm performs rather poorly when the signal is spiky, which is reflected by the fact that \eqref{boot1} performs especially poorly in the setting \eqref{spike}.  The method based on the bootstrap statistic \eqref{boot2} exceeds the nominal level significantly even though $\lambda(\mathcal{N})=0$ in all cases we consider. This is not particularly surprising as the mean differences in the settings \eqref{bump}-\eqref{spike} contain sets of positive measure where $d_1(t)$ is very close to 0 compared to the magnitude of the noise processes \eqref{light} and \eqref{heavy} which leads to significant finite sample bias for $\hat T_{n,\Delta}$. Our recommendation is therefore as follows: For small and moderate sample sizes one should use the methodology based on \eqref{boot3} while for large sample sizes (or heavy tailed data) usage of the method based on \eqref{boot1} is advisable as the gains in power compared to \eqref{boot3} are quite substantial (see table \ref{tab5}).

\begin{table}[H]
	\centering 
	\begin{tabular}{|c|c|c|c|c|c|c|c|c|c|c|}
	    \hline
        \multicolumn{5}{|c|}{Light Tails} & \multicolumn{4}{|c|}{Heavy Tails} \\
        \hline
	 Statistic   & \eqref{const} & \eqref{bump} &\eqref{bumps} & \eqref{spike} & \eqref{const} & \eqref{bump} &\eqref{bumps} & \eqref{spike} \\
		\hline     
    	\eqref{boot1} & 0.514 & 0.359 & 0.256 & 0.436 & 0.500 & 0.355 & 0.258 & 0.414  \\
	 \eqref{boot2} & 0.768 & 0.763 & 0.920 & 0.962 & 0.753 & 0.774 & 0.926 & 0.966    \\   
      \eqref{boot3}&  0.352 & 0.235 & 0.076 & 0.289 &  0.346 & 0.243 & 0.093 & 0.268\\ 
      \hline	
      \hline     
    	\eqref{boot1} &  0.709 & 0.478 & 0.430 & 0.513& 0.718 & 0.487 & 0.418 & 0.521 \\
	 \eqref{boot2} & 0.871 & 0.798 & 0.962 & 0.980 & 0.870 & 0.801 & 0.972 & 0.980 \\   
      \eqref{boot3}&  0.566 & 0.351 & 0.213 & 0.359 & 0.577 & 0.368 & 0.212 & 0.386\\ 
      \hline	
      \hline     
    	\eqref{boot1} &  0.978 & 0.760 & 0.837 & 0.742  & 0.969 & 0.741 & 0.878 & 0.749\\
	 \eqref{boot2} &   0.991 & 0.927 & 1.000 & 0.996 & 0.988 & 0.929 & 0.997 & 0.989\\   
      \eqref{boot3}&  0.927 & 0.643 & 0.641 & 0.591 & 0.925 & 0.647 & 0.629 & 0.626\\ 
      \hline	
	\end{tabular}
	\smallskip
	
	\caption{Empirical Rejection Rates of the bootstrap tests  for the hypotheses \eqref{h1} based on the bootstrap statistics \eqref{boot1}-\eqref{boot3}. The error processes are given by \eqref{light} in the light tailed and by \eqref{heavy} in the heavy tailed case. The upper part of the table contains the results for sample size $n=100$, the middle part for sample size $n=200$ and the lower part contains the results for sample size $n=500$. In all cases $\kappa=0.4$ and $\Delta$ is chosen such that $\Delta=d_1/2$, i.e. we are in the alternative.}
	\label{tab5}
\end{table}

\subsection{Real Data Application}
Just as \cite{dette2020} we follow \cite{Fremdt2014} and consider annual temperature curves of daily minimum temperatures from Melbourne, Australia. This yields 156 yearly temperature curves for the time 1856-2011. We proceed analogously to the synthetic data and estimate the bandwidth $l_n$ by the method from \cite{Rice2017} which yields $l_n=7$.\\

We detect a change point at $\hat s=0.67$ which corresponds to the year 1960 and the null hypothesis of equal means is rejected with a $p$ value below 0.01. To gain more insight we also consider the relevant hypotheses \eqref{h1} and use  the method in remark \ref{r1} to find a maximal $\Delta$ for which we still reject the null at level 0.05, we focus on the bootstrap procedure \eqref{boot3} as the sample size is moderate (see the discussion in the preceding subsection, we remark that the results for the other two procedures only differ by 0.1 degrees). We reject $H_0(\Delta)$ for all $\Delta<1.175$. In this context the  $L^1$ norm represents an average absolute mean minimum temperature difference, plotting the estimators (see figure \ref{Fig:3}) of $\mu^{(1)}$ and $\mu^{(2)}$ suggests that the temperature shift is purely upwards which in turn suggests that we have strong evidence for a mean minimum temperature shift of up to $1.175$ degrees Celsius. \\

\begin{figure}[H]
 
    \includegraphics[scale= 0.66]{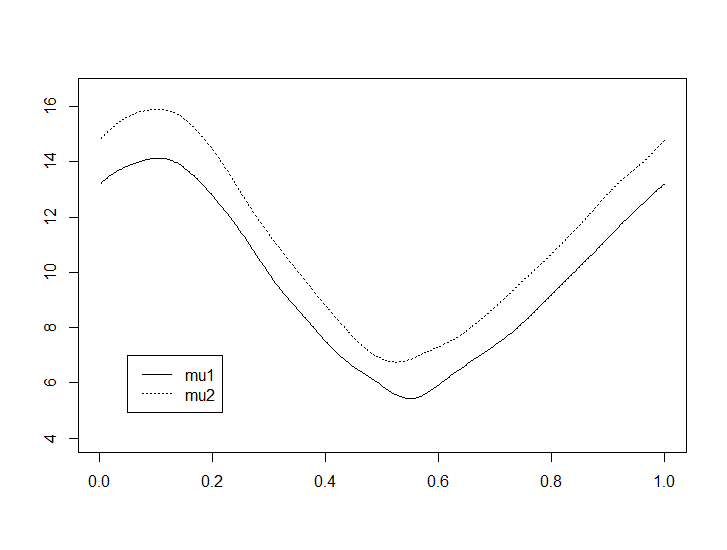}       
    
    \caption{\it  Plots of the mean estimator before (mu1) and after(mu2) the estimated change point for the Melbourne temperature time series. 
 }
    \label{Fig:3}
\end{figure}

Compared to the results in \cite{dette2020} who investigated \eqref{h1} with the supremum norm instead of the $L^1$ norm we note that the mean minimum temperature shift and the maximal minimum temperature shift are very close (they detected a maximal shift of  roughly $1.3$ degrees Celsius but used the smaller bandwidth $l_n=1$, we detect a mean shift of 1.27 degrees for this bandwidth) with the mean shift lagging slightly behind the maximum shift.  \\

In summary we observe strong evidence for a mean temperature shift in Melbourne that is of roughly the same size and lagging slightly behind the maximal temperature shift of about 1.3 degrees Celsius detected in \cite{dette2020}.\\

\textbf{Acknowledgements}
This research was funded in the course of TRR 391 Spatio-temporal Statistics for the Transition of Energy and Transport (520388526) by the Deutsche Forschungsgemeinschaft (DFG, German Research Foundation).

\pagebreak

\section{Proofs}
\label{sec7}
In this section we will often write, for a function $f \in C([0,1],L^1)$, the evaluation $(f(s))(t)$ as $f(s,t)$ to keep expressions readable.\\
When considering the setting \eqref{setting} we also assume that $\mu^{(1)}=0$ which can be achieved by a simple centering of the data which does not impact any of the results or proofs below except for making them easier to read. \\
Additionally all sums $\sum_{k=i_1}^{i_2}$ with $i_2<i_1$ are understood to be empty, i.e. equal to 0. For two sequences $a_n$ and $b_n$ we write
\begin{align}
    a_n \lesssim b_n
\end{align}
whenever $a_n \leq cb_n$ for some $c>0$ that does not depend on $n$.\\

\subsection{Some Facts about Cotype 2 Spaces}\hfill\\
We denote for a pregaussian $X \in L^1$ the associated Gaussian with covariance equal to that of $X$ by $G_X$. For the convenience of the reader we record a number of Lemmas that will be useful for the other proofs, they are either taken directly from \cite{ledoux1991} or are immediate consequences of the results therein.

\begin{Lemma}
\label{Lemma:Conc}
    Let $X$ be pregaussian and let $Y$ be a tight mean zero random variable in some Banach space $B$.  Suppose that for every $f \in B^*$ we have $\E[f(Y)^2]\leq \E[f(X)^2]$. Then $Y$ is also pregaussian and we have
    \begin{align}
        \E[\norm{G_Y}_B^p] \lesssim \E[\norm{G_X}_B^p]
    \end{align}
    for all $p>0$. 
\end{Lemma}

\begin{Lemma}
\label{Lemma:Cotype}
Let $B$ be a Banach space of cotype 2 and let $X \in B$ be a tight mean zero random variable that is pregaussian with associated Gaussian $G_X$. Then
\begin{align}
    \E[\norm{X}^2 ]\leq C\E[\norm{G_X}^2]
\end{align}
for some $C>0$ that only depends on $B$. Conversely, if the above holds for all pregaussian random variables, $B$ is of cotype 2.
\end{Lemma}

\begin{Lemma}
\label{Lemma:Shift}
 Let $X$ be a pregaussian random variable in $L^1$. Then
 \begin{align}
     G_{X(\cdot)-X(\cdot-y)}=G_{X}(\cdot)-G_X(\cdot - y)
 \end{align}
\end{Lemma}
\begin{proof}
We only need to show that the covariance operators are equal. To that end let $g,h \in L^\infty$ and consider
\begin{align}
    g(X(\cdot-y))=\int_0^1 g(t)X(t-y)dt=\int_0^1 g(t+y)X(t)=g_y(X(\cdot))
\end{align}
where $g_y$ denotes $g(\cdot+y)$ for any $g \in L^\infty$ (we extend $g$ to be 0 outside of $[0,1]$). Using this we straightforwardly obtain
\begin{align}
    g(X(\cdot)-X(\cdot-y))h(X(\cdot)-X(\cdot-y))&=g(X)h(X)-g_y(X)h(X)-g(X)h_y(X)+g_y(X)h_y(X)\\
    &=(g-g_y)(X)(h-h_y)(X)~.
\end{align}
Taking expectations and using that $X$ and $G_X$ have the same covariance operator yields
\begin{align}
    \E[g(X(\cdot)-X(\cdot-y))h(X(\cdot)-X(\cdot-y))]&=\E[(g-g_y)(G_X)(h-h_y)(G_X)]\\
    &=\E[g(G_X(\cdot)-G_X(\cdot-y))h(G_X(\cdot)-G_X(\cdot-y))]
\end{align}
which establishes the desired result.
\end{proof}

\subsection{Theorem \ref{Thm:L1Inv}: Strong invariance principle for $L^1$ valued time series}\hfill\\
\begin{proof}
    We want to apply Theorem 3 from \cite{Dehling1983}. To that end we need to find a suitable family of projections $P_N$ that have $N$ dimensional range and fulfill the approximation assumption 1.19 in the cited reference.\\    
    
    We define
    \begin{align}
        P_Nf(x)=N\sum_{i=1}^N\int_{(i-1)/N}^{i/N}f(z)dz\1\{x \in [(i-1)/n,i/n)\}
    \end{align}
    and note that it has norm 1 as an operator from $L^1$ to $L^1$. We further calculate
    \begin{align}
        \norm{Pf-f}_1&=\sum_{i=1}^N\int_{(i-1)/N}^{i/N}|f(x)-Pf(x)|dx\\
        &\leq \sum_{i=1}^N\int_{(i-1)/N}^{i/N}\Big|N\int_{(i-1)/N}^{i/N}f(x)-f(z)dz\Big|dx\\
        &\leq N\sum_{i=1}^N\int_{(i-1)/N}^{i/N}\int_{-1/N}^{2/N}|f(x)-f(x+y)|dydx        \\
        &\leq 3\sup_{|y| \leq 2/N}\int_{[0,1]}|f(x)-f(x+y)|dx  
    \end{align}
    where we used a change of variables in the third line and swapped integration order to obtain the last line. Now we square, replace $f$ by $\sqrt{n}S_n:=n^{-1/2}\sum_{j=1}^nX_j$ and take Expectations to obtain
    \begin{align}
        \E\Big[n\norm{S_n-P_NS_n}_1^2\Big] \lesssim \E\Big[\sup_{|y| \leq 2/N}n\norm{S_n(\cdot)-S_n(\cdot+y)}_1^2\Big]
    \end{align}
    We will now check the assumption of Theorem 2.2.4 in \cite{Wellner1996} to bound this quantity. As $L^1$ has cotype 2 we have that
    \begin{align}
        \E\Big[n\norm{S_n(\cdot)-S_n(\cdot+y)}_1^2\Big] \lesssim \E[\norm{G_{\sqrt{n}(S_n(\cdot)-S_n(\cdot+y))}}_1^2]
    \end{align}
    Now observe that for any $g \in (L^1)^*=L^\infty$ we have by the arguments leading to Theorem 3 in \cite{Yoshi78} that
    \begin{align}
        \E[g(\sqrt{n}(S_n(\cdot)-S_n(\cdot+y)))^2]\lesssim \E[g(\epsilon_1(\cdot)-\epsilon_1(\cdot+y))^2]
    \end{align}
    By Lemmas \ref{Lemma:Conc}, \ref{Lemma:Shift} and A4) we thus have 
    \begin{align}
        \E[\norm{G_{\sqrt{n}(S_n(\cdot)-S_n(\cdot+y))}}_1^2]& \lesssim \E[\norm{G_{\epsilon_1}(\cdot)-G_{\epsilon_1}(\cdot-y)}_1^2] \\
        & \lesssim y^{2\alpha}~.
    \end{align}    
   We may therefore apply Theorem 2.2.4 from \cite{Wellner1996} to obtain, for some $c>0$,
    \begin{align}
        \E\Big[\sup_{|y| \leq 2/N}n^{-1}\norm{S_n(\cdot)-S_n(\cdot+y)}_1^2\Big] \lesssim N^{-c}~.
    \end{align}
    With this we have checked assumption 1.19 from \cite{Dehling1983} which finishes the proof.
\end{proof}

\subsection{Theorem \ref{Thm:BootCons}: Weak invariance principle for the classical bootstrap}\hfill\\
The result immediately follows from the following weak invariance principle for the bootstrap process in combination with the continuous mapping theorem (see Lemma 2.2 in \cite{Bucher2019}).

We will verify, for any $k\geq 1$, the weak convergence (in $C([0,1],L^1)^{k+1}$)
\begin{align}
    \sqrt{n}(n^{-1}\sum_{i=1}^n\epsilon_i,S_{n,1}^*,...,S_{n,k}^*) \rightarrow (\W, \W_1,...,\W_k)
\end{align}
where $\W_i$ are iid copies of $\W$ and $S_{n,i}^*$ are bootstrap replicas of $S_n^*$ with mutually independent multipliers. To that end we will need to verify convergence of the finite dimensional distributions and establish tightness. The convergence of the finite dimensional distributions follows by exactly the same arguments as in the proof of Theorem 4.3 in \cite{dette2020}, where we replace evaluation at some point $t \in [0,1]$ by linear functionals on $L^1$. We will therefore only spell out the verification of tightness in all its details. As joint tightness is equivalent to marginal tightness we need only establish it for $S_n^*$ (tightness for the coordinate involving $(\epsilon_i)_{i=1,...,n}$ follows from Theorem \ref{Thm:L1Inv}). We first establish the desired result with  $\tilde S_n^*(s)=\frac{1}{n}\sum_{i=1}^{ns}\frac{v_i}{\sqrt{l}}\sum_{k=0}^{i+l-1}\epsilon_{n,i+k}$ instead and then show that the difference $\tilde S_n^*-S_n^*$ is asymptotically negligible. We also recall that we, WLOG, assume that $\mu^{(1)}=0$.\\

\textbf{Tightness}
We want to show that $\tilde S_n^*$ is a tight sequence of processes, i.e. we need to verify that $\p(\sqrt{n} \tilde S_n^* \in A)\geq 1-\epsilon$ for some compact subset $A$ of $C([0,1],L^1)$. To that end we note that by Arzelà-Ascoli we therefore need to find a set $A$ with $\p(A)\geq 1-\epsilon$ on which we have that
\begin{enumerate}
    \item[i)] $\sqrt{n}\tilde  S_n^*(s)$ is relatively compact (in $L^1$) for all $s$
    \item[ii)] $\sqrt{n}\tilde S_n^*(\cdot)$ is uniformly equicontinuous.
\end{enumerate}
\textbf{To show $i)$} we need to show (see \cite{Sud57}) that
\begin{align}
\label{p1}
    \sqrt{n}\sup_{|y|<\rho}\norm{\tilde S_n^*(s,\cdot)-\tilde S_n^*(s,\cdot+y)}_1 
\end{align}
goes to 0 as $\rho$ goes to 0. Using the same arguments as in the proof of Theorem \ref{Thm:L1Inv} we may obtain that
\begin{align}
    \E\left[n\sup_{|y|<\rho}\norm{\tilde S_n^*(s,\cdot)-\tilde S_n^*(s,\cdot+y)}_1^2 \right] \lesssim \E\left[\norm{G_{\sqrt{n}(\tilde S_n^*(s,\cdot)-\tilde S_n^*(s,\cdot+y))}}_1^2\right]
\end{align}
By the independence of the multipliers and the data we have by the arguments in the proof of Theorem 3 from \cite{Yoshi78} in combination with Lemmas \ref{Lemma:Conc} and \ref{Lemma:Shift} that
\begin{align}
    \E\left[\norm{G_{\sqrt{n}(\tilde S_n^*(s,\cdot)-\tilde S_n^*(s,\cdot+y))}}_1^2\right] \lesssim \E[\norm{G_{\epsilon_1(\cdot)-\epsilon_1(\cdot+y)}}_1^2] \lesssim y^{2\alpha}~.
\end{align}
Applying Theorem 2.2.4 from \cite{Wellner1996} then yields a set $A_1$ on which \eqref{p1} holds with probability $1-\epsilon/2$. \\

\textbf{To establish $ii)$} we have to show that
\begin{align}
\label{p2}
\sup_{|s-t|<\rho}\norm{\tilde S_n^*(s,\cdot)-\tilde S_n^*(t,\cdot)}_1
\end{align}
goes to $0$ as $\rho$ goes to 0. Using the same arguments as for establishing $i)$ we obtain
\begin{align}
    \E\left[n\sup_{|s-t|<\rho}\norm{\tilde  S_n^*(s,\cdot)-\tilde S_n^*(t,\cdot)}_1^2 \right] \lesssim \E\left[\norm{G_{\sqrt{n}(\tilde S_n^*(s,\cdot)-\tilde S_n^*(t,\cdot))}}_1^2\right]~.
\end{align}
Using that the sum $\sqrt{n}(\tilde S_n^*(s,\cdot)-\tilde S_n^*(t,\cdot))$ has at most $\lceil |t-s| \rceil$ summands one may proceed as for $i)$  to obtain that
\begin{align}
    E\left[\norm{G_{\sqrt{n}(\tilde S_n^*(s,\cdot)-\tilde S_n^*(t,\cdot))}}_1^2\right] \lesssim |t-s|~.
\end{align}
This yields a set $A_2$ on which \eqref{p2} holds with probability $1-\epsilon/2$. Defining $A=A_1 \cap A_2$ yields the desired set. \\

\textbf{Approximating $S_n^*$ by $\tilde S_n^*$}
 We are therefore left with showing that 
\begin{align}
\label{pb4}
    \sup_{s \in [0,1]}\norm{\sqrt{n}S_n^*(s,\cdot)-\sqrt{n}\tilde S_n^*(s,\cdot)}_1=o_\p(1)
\end{align}
We first consider the case that $H_1$  holds and will argue in two steps. But first we require some definitions. Let
\begin{align}
    \check \mu^{(1)}&=\frac{1}{k^*}\sum_{i=1}^{k^*}X_i\\
    \check \mu^{(2)}&=\frac{1}{n-k^*}\sum_{k^*+1}^n X_i\\
   \check Y_{n,i}&= X_{n,i}- (\check \mu^{(2)}-\check \mu^{(1)})\1\{i \geq  k^*\}
\end{align}
and define
\begin{align}
    \check S_n^*(s)=\frac{1}{n}\sum_{i=1}^{ns}\frac{\nu_i}{\sqrt{l}}\Big(\sum_{k=0}^{i+l-1}\check Y_{n,i+k}-\frac{l}{n}\sum_{j=1}^n\check Y_{n,j}\Big)
\end{align}
\textbf{Step 1: Replacing $S_n^*$ by $\check S_n^*$}\\
By Theorem \ref{Thm:CPCon} we have that $|k^*-\hat k|=O_\p(1)$. Therefore we have  (due to $k^*/N \rightarrow s^* \in (0,1)$), that
\begin{align}
\label{pb10}
    \hat \mu^{(1)}-\check \mu^{(1)}=\frac{k^*-\hat k}{\hat kk^*}\sum_{i=k^*}^{\hat k}X_i=O_\p(n^{-1})
\end{align}
and a similar argument yields the same result for $\mu^{(2)}$.

Similarly we have for all but $O_P(1)$ many indices that $\1\{i\geq k^*\}=\1\{i\geq \hat k\}$. In particular we have for all but $O_P(1)$ many indices that 
\begin{align}
    \norm{\check Y_{n,i}-\hat Y_{n,i}}_1 \leq 2\max_{i=1,2}\norm{\hat \mu^{(i)}-\check \mu^{(i)}}_1=O_\p(n^{-1})
\end{align}
Letting $B=\{i: (i+l-1<\min(\hat k, k^*)) \lor (i>\max(\hat k,k^*))\}$ a simple calculation using \eqref{pb10} then yields that
\begin{align}
    \sqrt{n}\sup_{s \in [0,1]}\norm{S_n^*-\check S_n^*}_1&\leq \sup_{s \in [0,1]}\sqrt{l/n}\norm{O_\p(n^{-1})\sum_{i=\hat k \land k^*}^{ns}\nu_i \1\{i \in B\}}_1\\
    & \quad \quad +\sup_{s \in [0,1]}\sqrt{1/n}\norm{\sum_{i=1}^{ns}\frac{\nu_i}{\sqrt{l}}\1\{i \notin B\}\sum_{k=0}^{l-1}\check Y_{n,i+k}-Y_{n,i+k}}_1 \\
    & \quad \quad + \sup_{s \in [0,1]}\norm{\frac{1}{n}\sum_{i=1}^{ns}\frac{\nu_i}{\sqrt{l}}\frac{l}{n}\sum_{j=1}^n(\hat Y_{n,j}-\check Y_{n,j})}_1
\end{align}
The first term on the right hand side is $o_p(1)$ by Hölder's inequality and the fact that $\sup_{s \in [0,1]}|\sum_{i=1}^{ns}\nu_i|=O_\p(\sqrt{n})$. The second term can be handled by a simple application of the triangle inequality because $|B^c|=O_\P(l)$. The third term can be shown to be of order $O_\p(\frac{\sqrt{l}}{n})$ by similar arguments.

\textbf{Step 2: Replacing $\check S_n^*$ by $\tilde S_n^*$}\\
We have
\begin{align}
    \check S_n^*(s)-\tilde S_n^*(s)=\frac{1}{n}\sum_{i=k^*}^{ns}\frac{\nu_i}{\sqrt{l}}\sum_{k=0}^{l-1}\Big(  \mu^{(2)}-\check \mu^{(2)} \Big)+\frac{1}{n}\sum_{i=1}^{ns}\frac{\nu_i}{\sqrt{l}}\sum_{k=0}^{l-1}\Big(\check  \mu^{(1)}-\frac{1}{n}\sum_{j=1}^n \check Y_{n,j}\Big)
\end{align}
Using Theorem \ref{Thm:L1Inv} yields that 
\begin{align}
    \norm{\check \mu^{(i)}-\mu^{(i)}}_1=O_\p(n^{-1/2})  \quad , i=1,2
\end{align}
so that we obtain
\begin{align}
    \sup_{s \in [0,1]}\sqrt{n}\norm{\check S_n^*(s)-\tilde S_n^*(s)}_1&\leq \frac{\sqrt{l}}{\sqrt{n}} \sup_{s \in [0,1]}\norm{O_\p(n^{-1/2})\sum_{i=1}^{ns}\nu_i }_1\\
    &\leq \sqrt{l/n}O_\p(1)=o_\p(1)
\end{align}
where we used Hölder's inequality and the fact that $\sup_{s \in [0,1]}|\sum_{i=1}^{ns}\nu_i|=O_\p(\sqrt{n})$.\\

In the case where $H_0$ holds \eqref{pb4} follows along similar lines. Here one uses that for any $\epsilon>0$ we can find a $\rho$ so that $\hat s$ will take values in the set $[\rho,1-\rho]$ with probability $1-\epsilon$. On this event one can use arguments similar to the ones above to obtain that
\begin{align}
    \sup_{s \in [0,1]}\norm{\sqrt{n}S_n^*(s,\cdot)-\sqrt{n}\tilde S_n^*(s,\cdot)}_1
\end{align}
is small. We omit the details as they are not particularly interesting, but we will  demonstrate how to find $\rho$. As $\W(s)-s\W(1)$ is a Gaussian process Theorem 4.4.1 from \cite{bogachev2015} yields that $\norm{\W(s)-s\W(1)}_{\infty,1}$ has a continuous distribution. Ignoring the trivial case of $C=0$ we therefore have, using standard results on the modulus of continuity of brownian motions, that we can first find $a$ and then $b$, each sufficiently small, so that that
\begin{align}
    \p(\norm{\W(s)-s\W(1)}_{\infty,1}>a)\geq 1-\epsilon\\
    \p(\max_{s \in [0,b]\cup[1-b,1]}\norm{\W(s)-s\W(1)}_{1}<a) \geq 1-\epsilon
\end{align}
By Theorem \ref{Thm:L1Inv} (and again Theorem 4.4.1 from \cite{bogachev2015} to obtain continuity for the distribution of the partial maximum) we can then find $n$ sufficiently large so that
\begin{align}
    \p(\sqrt{n}\norm{\U_n}_{\infty,1}>a)\geq 1-2\epsilon\\
    \p(\sqrt{n}\max_{s \in [0,b]\cup[1-b,1]}\norm{U_n}_{1}<a) \geq 1-2\epsilon
\end{align}
Letting $\rho=b$ we are done by the definition of $\hat s$.

\subsection{ Theorem \ref{Thm:CPCon}: Consistency of Change Point Estimator}\hfill\\
We will use Corollary 2 from \cite{Hariz2007}, we  therefore need to define an appropriate 
seminorm  $N$ on the space $\mathcal{M}$ of finite (signed) measures on $L^1$. To that end we define  for a constant $c$ to be chosen later the family 
\begin{align}
    \mathcal{F}_c:=\{\phi_{s,t,c}:L^1 \rightarrow \R | \phi_{s,t,c}(g)=\int_s^t g(x)\land c dx\}
\end{align}
 of integral functionals. We now define the seminorm
\begin{align}
    N_c(\nu)=\sup_{(s,t) \in [0,1]} \Big|\int_{L^1}\phi_{s,t,c}(x)d\nu(x)\Big|
\end{align}
and note that for $P=\p^{X_1}, Q=\p^{X_n}$ we have
\begin{align}
    \int_{L_1}\phi_{s,t,c}(x)d(P-Q)(x)&=\E[ \phi_{s,t,c}(X_1)-\phi_{s,t,c}(X_n)]\\
    &=\int_s^t \E[X_1\land c](x)-\E[X_n \land c](x)dx~.
\end{align}
If $\mu^{(1)}\neq \mu^{(2)}$ we can choose $c$ large enough so that for some $(s,t)$ we have
\begin{align}
    \int_s^t \E[X_1\land c](x)-\E[X_n \land c](x)dx>0
\end{align} which implies $N_c(P-Q)>0$. We therefore only need to verify Assumptions 1 and 2 from \cite{Hariz2007} to apply their Corollary 2. Note that by Hölder's inequality ($g(x)\land c$ is always a bounded function!) $\phi_{s,t,c}$ is Lipschitz in $(s,t)$ with respect to the metric $d((s,t),(u,v))=|s-u|+|u-v|$ on $[0,1]^2$. Theorem 2.7.11 from \cite{Wellner1996} then yields that Assumption 2 holds true. Assumption 1 is an easy consequence of our Assumptions A1) and A2).
\\

\subsection{Theorem \ref{Thm:BootTest}}\hfill\\
We only consider the case $d_1>0$. The case $d_1=0$ follows by similar but easier arguments. We will first establish the desired result for the statistic
\begin{align}
    T_{n,\Delta}=\sqrt{n}\Big(\sup_{s \in [0,1]}\norm{\U_n(s)}_1-s^*(1-s^*)\Delta\Big)
\end{align}
The same then follows for $\hat T_{n,\Delta}$ by noting that an application of Theorem \ref{Thm:CPCon} yields
\begin{align}
    (s^*(1-s^*)-\hat s(1-\hat s))\Delta=O_\p(n^{-1})~.
\end{align}
The remainder of the proof consists of the following two steps:
\begin{enumerate}
    \item[1)] Show that the process $\sqrt{n}\Big(\U_n(s)-(s\land s^*-ss^*)(\mu^{(1)}-\mu^{(2)})\Big)_{s \in [0,1]}$ converges in distribution to $\Big(\W(s)-s\W(1)\Big)_{s \in [0,1]}$.
    \item[2)] Use the delta method to establish the result in the case $\Delta=d_1$. Derive the other two cases as corollaries to this case.
\end{enumerate}

\textbf{Step 1:} 
The result follows by noting that $\Big(\U_n(s)-(s \land s^*-ss^*)(\mu^{(1)}-\mu^{(2)})\Big)_{s \in [0,1]}$ is simply the sequential cusum process associated to the centered random variables $Z_i=X_i-\mu_i$ and an application of Theorem \ref{Thm:L1Inv} followed by an application of the continuous mapping theorem.\\
\textbf{Step 2:}
By Theorem \ref{Thm:HadDiff} and the delta method (see Theorem 2.1 in \cite{shapiro1991}) we obtain that
\begin{align}
    \sqrt{n}\Big( \norm{\U_n(s)}_{\infty,1}-s^*(1-s^*)d_1) \Big)\overset{d}{\rightarrow} T
\end{align}
which yields the desired result when $d_1=\Delta$. For the other cases we simply write
\begin{align}
    T_{n,\Delta}= \sqrt{n}\Big( \norm{\U_n(s)}_{\infty,1}-s^*(1-s^*)d_1) \Big) + \sqrt{n}s^*(1-s^*)(d_1-\Delta)
\end{align}
and notice that the first summand is tight while the second summand diverges to $\pm \infty$ depending on whether or not $d_1$ is larger or smaller than $\Delta$.\\

\subsection{Theorem \ref{Thm:RelBoot}}\hfill\\
Let $T_1,...,T_k$ be iid copies of $T$ and let $\hat T^*_1,...,\hat T^*_k$ be given by $\hat T^*$ calculated from $k$ bootstrap samples of $S_n^*$ which we denote by $S_{n,i}^*$. The first part of the theorem follows immediately from establishing, for all $k \geq 1$, the weak convergence
\begin{align}
    (\hat T, \hat T^*_1,...,\hat T^*_k) \rightarrow (T, T_1, ..., T_k)
\end{align}
where 
\begin{align}
    \hat T=\sqrt{n}\Big( \norm{\U_n(s)}_{\infty,1}-\hat s(1- \hat s)d_1) \Big)~.
\end{align}
Confer Lemma 2.2 from \cite{Bucher2019} for details on how to obtain the conditional convergence from this result.\\

We begin by noting that by Lemma \ref{Lem:SetEst}
\begin{align}
    \Big|\int_{\mathcal{N}}|\U_n^*(\hat s,t)|dt-\int_{\hat{\mathcal{N}}}|\U_n^*(\hat s,t)|dt\Big| &\leq \norm{\U_n^*(\hat s,\cdot)}_1\Big(\lambda(\hat{\mathcal{N}}\setminus \mathcal{N}) + \lambda(\mathcal{N}\setminus \hat{\mathcal{N}})\Big)\\
    &=o_\p(n^{-1/2})~,
\end{align}
where $\U_{n,i}^*(s)=S_{n,i}^*(s)-sS_{n,i}^*(1)$. A similar argument for the integral over $\mathcal{N}^c$ therefore yields that
\begin{align}
   \hat T_i^*&=\sqrt{n}\Big(\int_{\mathcal{N}^c}\text{sgn}(d(t))\U^*_{n,i}(\hat s,t)dt+\int_{\mathcal{N}}\Big|\U^*_{n,i}(\hat s,t)\Big|dt\Big)+o_\p(1)\\
   &=\sqrt{n}\Big(\int_{\mathcal{N}^c}\text{sgn}(d(t))\U^*_{n,i}(s^*,t)dt+\int_{\mathcal{N}}\Big|\U^*_{n,i}( s^*,t)\Big|dt\Big)+o_\p(1)~,
\end{align}
where the second line follows by Theorem \ref{Thm:CPCon} and some straightforward bounds.

By Theorem 2.1 from \cite{shapiro1991} and Theorem \ref{Thm:HadDiff} we also have that 
\begin{align}
    \hat T=\sqrt{n}\Big(\int_{\mathcal{N}^c}\text{sgn}(d(t))\V_n(s^*,t)dt+\int_{\mathcal{N}}\Big|\V_n(s^*,t)\Big|dt\Big)+o_\p(1)
\end{align}
where $\V_n(s)=\U_n(s)-(s \land s^*-ss^*)(\mu^{(1)}-\mu^{(2)})$. We remark that we can therefore write $(\hat T, \hat T^*_1,...,\hat T^*_k)$ as the sum of $o_\p(1))$ terms and a term that is a continuous function of
\begin{align}
\label{pb5}
    (\V_n, \U_{n,1}^*,...,\U_{n,k}^*)~.
\end{align}
An elementary calculation shows that one may apply Theorem \ref{Thm:L1Inv} to obtain that $\V_n$ converges weakly to $\W$. The weak convergence of $\U_{n,1}^*$ to $\W$ has been established in the proof of  Theorem \ref{Thm:BootCons}. The vector
\eqref{pb5} is therefore tight. Finite dimensional convergence also follows by the same arguments as in the proof of Theorem \ref{Thm:BootCons}. As a consequence the vector \eqref{pb5} converges in distribution to $(\W,\W_1,...,\W_k)$ where $\W_i$ are iid copies of $\W$. The desired result then follows by the continuous mapping theorem and the remark preceding equation \eqref{pb5}.
\\

The second part of the theorem follows by a simple case distinction between $d_1>0$ and $d_1=0$. In the former case the result follows immediately from  Theorem \ref{Thm:BootTest} and
\begin{align}
    \hat T^*\overset{d}{\rightarrow} T~.
\end{align}
The latter case follows from $\hat T^*$ being a tight random variable and Theorem \ref{Thm:BootTest}.

\begin{Lemma}
\label{Lem:SetEst}
    Grant assumptions A1) to A4). Then
    \begin{align}
        \Big(\lambda(\hat{\mathcal{N}}\setminus \mathcal{N}) + \lambda(\mathcal{N}\setminus \hat{\mathcal{N}})\Big)=o_\p(1)
    \end{align}
\end{Lemma}
\begin{proof}
    We know by equation \eqref{pb10} that
    \begin{align}
        \sqrt{n}\norm{\hat \mu^{(1)}-\hat \mu^{(2)} - (\mu^{(1)}-\mu^{(2)})}_1=O_\p(1)
    \end{align}
    Consequently
    \begin{align}
       \lambda\Big(\mathcal{N}\setminus \hat{\mathcal{N}}\Big)& \leq \lambda\Big(|\hat \mu^{(1)}-\hat \mu^{(2)}|>\frac{\log(n)}{\sqrt{n}}, |\mu^{(1)}-\mu^{(2)}|=0\Big)\\
       &\leq \frac{\sqrt{n}}{\log(n)}\int_0^1 |\hat \mu^{(1)}-\hat \mu^{(2)}|\1\{|\mu^{(1)}-\mu^{(2)}|=0\}d\lambda \\
       &=o_\p(1)
    \end{align}
    Similarly we have for any $c>0$ and $n$ sufficiently large that
    \begin{align}
        \lambda\Big(|\hat \mu^{(1)}-\hat \mu^{(2)}|<\frac{\log(n)}{\sqrt{n}}, |\mu^{(1)}-\mu^{(2)}|\geq c\Big)&\leq \lambda\Big(|\hat \mu^{(1)}-\hat \mu^{(2)} - (\mu^{(1)}-\mu^{(2)})|\geq c/2\Big)        
    \end{align}
    which yields 
    \begin{align}
        \lambda(\hat{\mathcal{N}}\setminus \{|\mu^{(1)}-\mu^{(2)}|< c\})=o_\p(1)
    \end{align}
    by Markovs inequality. Observing that for any $\eta>0$ we may choose $c$ small enough so that
    \begin{align}
        \lambda(0<|\mu^{(1)}-\mu^{(2)}|<c)<\eta
    \end{align}
    we obtain that
    \begin{align}
         \lambda(\hat{\mathcal{N}}\setminus \mathcal{N})\leq o_\p(1)+\eta
    \end{align}
    As $\eta$ was arbitrary we are done.
\end{proof}

\subsection{Directional Hadamard Differentiability of $\norm{\cdot}_{\infty,1}$}\hfill\\
\begin{theorem}
\label{Thm:HadDiff}
    The function
    \begin{align}
        \norm{\cdot}_{\infty,1}:C([0,1],L^1)&\rightarrow \R\\
         G &\rightarrow \norm{G}_{\infty,1}
    \end{align}
    is directionally hadamard differentiable at every $G \neq 0$. Its derivative at $G$ is given by
    \begin{align}
        D_G\norm{\cdot}_{1,\infty}:C([0,1],L^1)&\rightarrow \R\\
        H \rightarrow \sup_{s \in \mathcal{E}}\Big(\int_{G(s,\cdot)\neq 0} \text{sgn}(G(s,x))H(s,x)dx&+\int_{G(s,\cdot)=0}|H(s,x)|dx\Big)
    \end{align}
    where 
    \begin{align}
        \mathcal{E}=\Big\{s \Big| \int_0^1|G(s,x)|dx=\norm{\int_0^1|G(\cdot,x)|dx}_\infty \Big\}
    \end{align} 
\end{theorem}
\begin{proof}
    We will establish directional Hadamard differentiability of the following three mappings.
    \begin{align}
      \phi_1:C([0,1],L^1) &\rightarrow C([0,1],L^1)\\
        G &\rightarrow |G|\\
        \phi_2:C([0,1],L^1)& \rightarrow C([0,1],\R)\\
        G & \rightarrow \Big( s \rightarrow \int_0^1 G(s,x)dx \Big)\\
        \phi_3:C([0,1],\R)&\rightarrow \R\\
        f &\rightarrow \norm{f}_\infty
    \end{align}
    The result then follows by an application of the chain rule. (See Proposition 3.6 in \cite{Shapiro1990}). The differentiability of $\phi_2$ is obvious as it is linear and bounded. The differentiability of $\phi_3$ is shown in \cite{Carcamo2020}. We therefore only need to consider $\phi_1$. We will show that the directional hadamard derivative at $G$ is given by
    \begin{align}
        D_G\phi_3:C([0,1],L^1)& \rightarrow C([0,1],L^1)\\
             H & \rightarrow \text{sgn}(G)H\1\{|G|>0\}+|H|\1\{G=0\}
    \end{align}
    To that end we will proceed by first establishing Gateaux directional differentiability which, by Lipschitz continuity of $\phi_3$, is equivalent to Hadamard directional differentiability. This in turn we establish by first showing that 
    \begin{align}
    \label{pb2}
        W_n(s)=\norm{\frac{|G(s,\cdot)+tH(s,\cdot)|-|G(s,\cdot)|}{t}-D_G\phi_3(H)(s,\cdot)}_1
    \end{align}
    converges pointwise to 0. We then show that the family of functions \eqref{pb2} (indexed in $t$) is equicontinuous in $s$ (and thereby relatively compact by Arzela Ascoli). This yields that the convergence is uniform which then gives the desired result.

    \textbf{Pointwise Convergence:}
    Fixing $s$ we are interested in the asymptotic behaviour of 
    \begin{align}
        t^{-1}\int_0^1 \Big| |G(s,x)+tH(s,x)|-|G(s,x)|-D_G\phi_3(H)(s,x)\Big|dx~.
    \end{align}
    We will show that for any sequence $t_n \rightarrow 0$ the sequence 
    \begin{align}
        Z_n(s,x)=t_n^{-1}\Big(|G(s,x)+tH(s,x)|-|G(s,x)|\Big)-D_G\phi_3(H)(s,x)
    \end{align}
    converges, for each  $s$, to 0 locally in measure and is uniformly integrable which then yields the desired pointwise convergeence statement for \eqref{pb2}. Uniform integrability follows from applying the triangle inequality to obtain
    \begin{align}
        Z_n(s,x) \leq 2|H(s,x)|~.
    \end{align}
    That $Z_n(s,x)$ converges to 0 almost everywhere for each fixed $s$ follows from a simple case distinction.
    \\    
    \textbf{Equicontinuity}
    We have by the (reverse) triangle inequality that
    \begin{align}
    \label{pb3}
        |W_n(s)-W_n(u)|&\leq  \norm{H(s,\cdot)-H(u,\cdot)}_1 + \norm{D_G\phi_3(H)(s,\cdot)-D_G\phi_3(H)(u,\cdot)}_1
    \end{align}
    Note that $s \rightarrow H(s)$ and $s \rightarrow D_G\phi_3(H)(s)$ are continuous functions on a compact set which are therefore uniformly continuous. This yields the desired equicontinuity of $W_n(s)$ by upper bounding \eqref{pb3} by a joint modulus of continuity of the two functions.
\end{proof}

\subsection{Theorem \ref{Thm:PowImp}}\hfill\\
If the conditions of Theorem \ref{Thm:BootCons} are satisfied we only need to show that $J=0$ with high probability when $H_0$ holds. For that it suffices to establish (under $H_0$) the weak convergence
\begin{align}
    \sqrt{n}\U_n \rightarrow \W
\end{align}
in $C([0,1],C([0,1])\simeq C([0,1]^2)$. This is a straightforward consequence of a strong invariance principle for the sums $S_m=\sum_{i=1}^m\epsilon_i$, the properties of brownian motion and the continuous mapping theorem. To obtain such an invariance princple we want to apply Theorem 6 from \cite{Dehling1983}. We adopt their notation from pages 399 and 400 and choose $S=([0,1],|\cdot|^\alpha)$ so that $g(\epsilon)=1/\epsilon^\alpha$. We hence need to show that
    \begin{align}
    \label{p7}
        \sup_{m \in \N}\E\left[n^{-1}\sup_{|t-t^\prime|^\alpha<1/n^\alpha}|S_m(t)-S_m(t^\prime)|^2\right]\lesssim n^{-s}
    \end{align}
    for some $s>0$. We want to apply Theorem 2.2.4 from \cite{Wellner1996} to bound the left hand quantity. To be able to apply Theorem 2.2.4 we need to show that
    \begin{align}
    \label{equicont}
        n^{-1/2}\E[|S_m(t)-S_m(t^\prime)|^2]^{1/2}\leq K_1 |t-t^\prime|^\alpha~,
    \end{align}
    which follows by Assumptions (A2) and (A3) and the arguments used for the proof of Theorem 3 from \cite{Yoshi78}. Theorem 2.2.4 then yields that for any $\nu>0$ and some $K_2>0$ depending only on $K_1$ we have
    \begin{align}
        \E\left[n^{-1}\sup_{|t-t^\prime|^\alpha<1/n^\alpha}|S_m(t)-S_m(t^\prime)|^2\right]^{1/2}
        &\leq K_2 \Big( \int_0^\nu \epsilon^{-1/(2\alpha)}d\epsilon+n^{-\alpha} \nu^{-2/(2\alpha)}\Big) \\
        &=K_2 \Big(\frac{\nu^{1-1/(2\alpha)}}{1-1/(2\alpha)}+n^{-\alpha}\nu^{-2/(2\alpha)}\Big)~.
    \end{align}
    Choosing, for instance, $\nu=n^{-\alpha/4}$ yields \eqref{p7} and finishes the proof.\\

    We are left with showing that we may apply Theorem \ref{Thm:BootCons} for which conditions A3) and A4) are missing. As the above invariance principle also applies to an iid sequence of $\epsilon_i$ there exists a Gaussian random variable taking values in $C[0,1] \subset L^1$ with the same covariance operator as $\epsilon_i$. This yields that $\epsilon_i$ is pregaussian in $C[0,1]$ (and hence also $L^1$) and therefore that A3) is satisfied. Equation \eqref{equicont} also yields that 
    \begin{align}   
        n^{-1/2}\E[|G(t)-G(t^\prime)|^2]^{1/2}\leq K_1 |t-t^\prime|^\alpha~,
    \end{align}
    where $G$ is the weak limit of $n^{-1/2}S_n$. This immediately yields A4)
\subsection{Theorem \ref{Thm:Alter}}\hfill\\
\begin{proof}
In the proof of Theorem \ref{Thm:PowImp} we obtain an  invariance principle for $C([0,1])$ valued data under the conditions of this theorem. This also yields invariance principles in the spaces $L^1$ and $L^2$. The result for the $L^1$ norm then follows by the proof of Theorem \ref{Thm:BootTest}. The proofs for the $L^2$ and supremum norm follow along similar lines, are simpler and are therefore omitted. The necessary directional Hadamard derivatives are either easy to obtain ($L^2$ case) or available from \cite{Carcamo2020}.
\end{proof}

\bibliographystyle{apalike}

\bibliography{main}
\end{document}